\documentclass{amsart}

\usepackage{latexsym,amsmath,amsthm}
\usepackage{thmtools}
\usepackage{amssymb}
\usepackage{dsfont}
\usepackage[pdftex,colorlinks,backref=page,citecolor=blue]{hyperref}
\usepackage{geometry}
\usepackage{comment}
\usepackage[pdftex]{graphicx}

\usepackage{graphpap, color, paralist, pstricks}

\usepackage{setspace}
\setdisplayskipstretch{1}
\usepackage{mathtools}
\usepackage{cleveref}

\newtheorem{theorem}{Theorem}[section]
\newtheorem{proposition}[theorem]{Proposition}
\newtheorem{corollary}[theorem]{Corollary}
\newtheorem{lemma}[theorem]{Lemma}

\newtheorem{claim}[theorem]{Claim}

\newtheorem{fact}[theorem]{Fact}
\newtheorem{definition}[theorem]{Definition}
\newtheorem{observation}[theorem]{Observation}

\newcommand{\gam}{\gamma}
\newcommand{\Gam}{\Gamma}
\newcommand{\kap}{\kappa}
\newcommand{\lam}{\lambda}
\newcommand{\Lam}{\Lambda}
\newcommand{\gO}{\Omega}

\newcommand{\PP}{\mathbb{P}}

\newcommand{\cA}{\mathcal{A} }

\newcommand{\cC}{\mathcal{C} }

\newcommand{\cG}{\mathcal{G} }
\newcommand{\cH}{\mathcal{H} }

\newcommand{\cP}{\mathcal{P} }

\newcommand{\cV}{\mathcal{V} }
\newcommand{\cW}{\mathcal{W} }

\newcommand{\bT}{\mathbf{T}}

\newcommand{\beq}[1]{\begin{equation}\label{#1}}
	\newcommand{\enq}[0]{\end{equation}}

\newcommand{\gd}[0]{\delta }

\newcommand{\gk}{\kappa}

\newcommand{\nin}[0]{\noindent}

\newcommand{\ov}[0]{\overline}
\newcommand{\sub}[0]{\subseteq}

\newenvironment{subproof}[1][\proofname]{
  
  \begin{proof}[#1]
}{
  \end{proof}
}

\newcommand{\E}{\mathbb{E}}

\newcommand{\dist}{{\textrm{dist}}}

\newcommand{\al}{\alpha}
\newcommand{\ka}{\kappa}
\newcommand{\lex}{<_{\textrm{lex}}}

\begin{document}

\title{On the number of Antichains in $\{0,1,2\}^n$}

\author[M. Jenssen]{Matthew Jenssen}
\address{Department of Mathematics, King's College London}
\email{matthew.jenssen@kcl.ac.uk}

\author[J. Park]{Jinyoung Park}
\address{Department of Mathematics, Courant Institute of Mathematical Sciences, New York University}
\email{jinyoungpark@nyu.edu}

\author[M. Sarantis]{Michail Sarantis}
\address{Institute of Discrete Mathematics, Graz University of Technology}
\email{msarantis@tugraz.at}

\begin{abstract}
We provide precise asymptotics for the number of antichains in the poset $\{0,1,2\}^n$, answering a question of Sapozhenko. 
Finding improved estimates for this number was also a problem suggested by Noel, Scott, and Sudakov, who obtained 
asymptotics for the logarithm of the number. Key ingredients for the proof include a graph-container lemma to bound the number of expanding sets in a class of irregular graphs, isoperimetric inequalities for generalizations of the Boolean lattice, and methods from statistical physics based on the cluster expansion. 
\end{abstract}

\maketitle

\section{Introduction}\label{sec.intro}

An \textit{antichain} of a poset $P$ is a subset of $P$ whose elements are pairwise incomparable. We denote by $\al(P)$ the total number of antichains in $P$. In 1897, Dedekind \cite{Dedekind} asked for the number of antichains in the Boolean lattice $\{0,1\}^n$, which later became known as \textit{Dedekind's problem}.

Since the middle layer of $\{0,1\}^n$ is an antichain, we have the simple lower bound $\alpha(\{0,1\}^n) \ge 2^{\binom{n}{\lfloor n/2\rfloor}}$. Kleitman \cite{Kleitman1969OnDP} and subsequently Kleitman and Markowsky \cite{Kleitman1975OnDP} proved this bound is asymptotically correct on the logarithmic scale, namely $\log\alpha(\{0,1\}^n)\leq \left(1+O(\log n/n)\right)\binom{n}{\lfloor n/2\rfloor}$. (Throughout the paper, $\log$ denotes $\log_2$ and $\ln$ denotes the natural logarithm.) This result was subsequently recovered by Pippenger \cite{Pippenger1999EntropyAE} (with a weaker error term) and Kahn \cite{Kahn2002Entropy} using (distinct) entropy tools. Korshunov \cite{korshunov} was the first to estimate $\alpha(\{0,1\}^n)$ itself up to a factor of $(1+o(1))$; the proof is technical, and it was later simplified (while still being involved) by Sapozhenko \cite{Sapozhenko1989}. Sapozhenko's proof was an early application of  what is now known as the \textit{graph container method}.

Note that the Boolean lattice $\{0,1\}^n$ is the Cartesian product of $n$ chains on two elements. Sapozhenko~\cite{Sapozhenko_proceedings} raised the question of extending the container method used in the solution of Dedekind’s problem to the poset $\{0,1,2\}^n$, that is, the product of $n$ chains on three elements. He pointed out that the principal difficulty lies in the strong irregularity of this poset, in contrast with the Boolean lattice.  The question of improving estimates on $\al(\{0,1,2\}^n)$ was also asked by Noel, Scott and Sudakov \cite[Section 7]{NSS}. In their work, they proved supersaturation results for comparable pairs in $\{0,1,2\}^n$ and combined them with a version of the graph container method to estimate $\log \alpha(\{0,1,2\}^n)$ up to a factor of $(1+O(\sqrt{(\log n)/n}))$. In recent years, there has been a renewed interest in enumerating antichains in $\{0,\ldots,t-1\}^n$, the product of $n$ chains of length $t$, see e.g. \cite{Carroll2012CountingAA, falgasravry2023dedekinds, park2023notenumberantichainsgeneralizations, Pohoata2021}. All of these results yield asymptotics on the logarithmic scale, that is, an estimate of $\log \alpha(\{0,\ldots,t-1\}^n)$ up to a multiplicative error of $(1+o(1))$. Our main result provides an estimate for $\alpha(\{0,1,2\}^n)$ itself up to a multiplicative error of $(1+o(1))$. We also describe in detail the typical structure of an antichain in $\{0,1,2\}^n$.
  
  We note that our methods can be used to obtain precise asymptotics for other posets including the collection $\cV(q,n)$ of subspaces of $\mathbb F^n_q$ ordered by set inclusion, which was another problem suggested in \cite{NSS}; this is a straightforward application of the current methods so we do not pursue the details here.

We now briefly explain why Sapozhenko's approach for estimating $\alpha(\{0,1\}^n)$ doesn't easily extend to $\alpha(\{0,1,2\}^n)$. Let $\Sigma$ be a bipartite graph with parts $X$ and $Y$, and let $d_X$ be the minimum degree of vertices in $X$ and $D_Y$ be the maximum degree of vertices in $Y$. A key step in Sapozhenko's approach is to reduce the problem of bounding $\alpha(\{0,1\}^n)$ to bounding the number of independent sets in the bipartite graphs induced on two consecutive layers of the Boolean lattice. In this setting, Sapozhenko's argument crucially assumes
\beq{Sap.cond} d_X \ge D_Y.\enq Note that any two consecutive layers of $\{0,1\}^n$ form a biregular bipartite graph, and it satisfies the above degree condition (choosing $X$ to be the smaller layer).
On the other hand, the bipartite graphs induced on two consecutive layers of $\{0,1,2\}^n$ do not necessarily satisfy the above degree condition. 
As Sapozhenko highlights \cite{Sapozhenko_proceedings}, the irregular structure of $\{0,1,2\}^n$ is the main obstacle for the application of his method. One of the main contributions of our work is to prove a container type lemma which goes beyond the degree assumption in \eqref{Sap.cond}. It therefore may be of independent interest, both for future applications  and as progress toward a container lemma for graphs under optimal (minimal) degree assumptions.

\subsection{Main results and strategy}
We start with some basic definitions. For brevity, we henceforth let $[t]^n=\{0, 1, \ldots, t-1\}^n$ denote the poset that is the product of $n$ chains on $t$ elements. Equivalently, it is the poset of all $n$-tuples $(x_1,\dots,x_n)$ of integers with $0\leq x_i\leq t-1$ for all $i$, equipped with the partial order $x\le y$ iff  $x_i\leq y_i$ for all $i$. For $v \in [t]^n$ and $i\in [n]$, write $v_i$ for the $i$-th coordinate of $v$. For $k \in \{0,1, \dots,(t-1)n\}$, $L_k(t,n)$ denotes the \textit{$k$-th layer} of $[t]^n$, that is,
\[L_k(t,n):=\left\{v \in [t]^n: \sum_{i \in [n]} v_i=k\right\}.\]
Write $\ell_k(t,n)$ for the size of $L_k(t,n)$, and $L_{[a,b]}(t,n):= \cup_{a \le k \le b} L_k(t,n)$. We let $\ell(t,n)=\max_k \ell_k(t,n)$, the maximum size of a layer in $[t]^n$.
We will often suppress the dependency on $t$ or $n$ (or both) from this notation, if the context is clear. 

We now discuss our main results. Henceforth we let  $\ell_n=\ell(3,n)$, that is, the size of the (unique) middle layer of $[3]^n$.

\begin{theorem}\label{thm:asymp_1}
        \begin{align*}
            \alpha([3]^n)=\left(1+e^{-\gO(n)}\right)2^{\ell_n}\exp{(T_1+T_2)}
        \end{align*}
        where
        $$T_1=2\sum_{{0\leq k<n/2}} {n \choose 2k+1}{2k+1 \choose k}2^{-(n-k)}$$
        and
        $$T_2=\sum_{{0\leq k<n/2}}\binom{n}{2k+1}\binom{2k+1}{k}\left(n^2-3(k+1)n+\frac{k(9k+17)}{4}\right)2^{-2(n-k)}\, .$$
\end{theorem}

We note that Falgas-Ravry, R\"aty, and Tomon \cite{falgasravry2023dedekinds} showed that (for any $t$)
\begin{equation}\label{eq:FRT_conj}\log_2\alpha([t]^n)\ge \left(1+2^{-\Theta(n)}\right)\cdot {\ell(t,n)},\end{equation}
and posited that ``it might be reasonable to conjecture that this lower bound is closer to optimal" (than their upper bound). \Cref{thm:asymp_2} confirms this (tentative) conjecture in a strong form for $t=3$. Indeed, by computing the asymptotics of $T_1$ and showing that $T_2=o(T_1)$, we obtain the following corollary.

\begin{theorem}\label{thm:asymp_2} 
    $$\al([3]^n)=2^{\ell_n}\exp\left[(1+o(1))\sqrt{\frac{1+2\sqrt{2}}{2\sqrt{2}\pi}}n^{-1/2}\left(\frac{1+2\sqrt{2}}{2}\right)^n\right].$$
\end{theorem}

At a high level, we follow Sapozhenko's approach for Dedekind's problem: \Cref{thm:asymp_1} will follow from a series of results that detail the structure of a typical antichain of $[3]^n$. First, we will prove that almost all antichains are contained in the three ``central" layers $L_{[n-1, n+1]}$ (\Cref{thm.reduction}). We will then show that almost all antichains in these layers have a very specific structure (\Cref{thm.central}). This allows us to count them and get the asymptotic expression of \Cref{thm:asymp_1}. 

    \begin{theorem}\label{thm.reduction}
        \[\alpha([3]^n)=\left(1+e^{-\gO(n^2)}\right)\alpha(L_{[n-1, n+1]}).\]
    \end{theorem}

    We say a set of vertices in a graph $G$ is \textit{$2$-linked} if it is connected in $G^2$ (the "square" of $G$, i.e., the graph where every two vertices of $G$ are adjacent if they have distance at most $2$ in $G$). A set $W\subseteq V(G)$ is called a \textit{2-linked component} if $W$ is a connected component of the graph $G^2$. We say that a set $S\subseteq [3]^n$ is $2$-linked if it is $2$-linked in the Hasse diagram of $[3]^n$ ($x,y$ are adjacent if $y<x$ and there doesn't exist $z$ such that $y<z<x$).

    \begin{theorem}\label{thm.central} Let $\cA_0$ be the collection of the antichains $I$ of $[3]^n$ that satisfy the following properties:
\begin{enumerate}[(i)]
    \item $I \sub L_{[n-1, n+1]}$;
    \item every 2-linked component of $I \backslash L_n$ has size at most 2.
\end{enumerate}
Then
\begin{align*}
            \alpha(L_{[n-1, n+1]})&=\left(1+e^{-\gO(n)}\right)|\cA_0|
        \end{align*}
    and, with $T_1, T_2$ as in \Cref{thm:asymp_1},
    $$\alpha(L_{[n-1, n+1]})=\left(1+e^{-\gO(n)}\right)2^{\ell_n}\exp{(T_1+T_2)}.$$
    \end{theorem}
Given an antichain $I \sub [3]^n$, call the vertices of $I\backslash L_n$ the \textit{defect vertices} of $I$.  
Let $X_n$ be the random variable that counts the number of defect vertices in a uniformly randomly chosen antichain in $[3]^n$.
Let $\mu_n=\E X_n$ and $\sigma_n^2={\rm Var}(X_n).$ We prove a central limit theorem for $X_n$.

\begin{theorem}\label{thm:CLT}
    The random variable $\tilde X_n:=\frac{X_n-\mu_n}{\sigma_n}$
    converges in distribution to the standard normal $N(0,1)$ as $n \rightarrow \infty.$ Furthermore, 
    $$\mu_n=(1+o(1)){T_1} \hspace{1cm} \text{and} \hspace{1cm} \sigma_n^2=(1+o(1)){T_1}.$$
\end{theorem}

 \subsection{Connection to high-dimensional partitions and a Ramsey-type problem}
 Beyond being a natural generalization of Dedekind's problem, antichains in $[t]^n$ have received considerable attention because of their connections to several other combinatorial problems. We give a couple of examples here. 
 An $n$-dimensional integer partition of an integer $t$ is a $t\times t\dots\times t$, $n$-dimensional tensor $A$ of non-negative integers that sum up to $t$, satisfying $A_{i_1,\dots,i_{k},\dots,i_n}\ge A_{i_1,\dots,i_{k}+1,\dots,i_n}$ for every possible $i_1, \ldots, i_n$ and all $1\leq k \leq n$. Let $P_n(t)$ be the number of $t\times t\dots\times t$ $n$-dimensional partitions with entries in $\{0,1,\dots, t\}$. A straightforward bijection \cite[Observation 2.5]{Moshkovitz2012RamseyTI} shows $P_{n-1}(t)=\alpha([t]^n)$.
    
    Another connection is given by the following Ramsey-theoretic problem of `Erd\H{o}s-Szekeres type', first defined and studied by Fox, Pach, Sudakov and Suk in \cite{Fox2011ErdsSzekerestypeTF}. For any sequence of positive integers $x_1<x_2<\ldots<{x_{l+k-1}}$, we say that the $k$-tuples $(x_i,x_{i+1}, \ldots, x_{i+k-1})$ $(i=1,2,\ldots, l)$ form a \textit{monotone path} of length $l$.
    Let $M_k(t,n)$ be the smallest integer $M$ such that any coloring of the $k$-element subsets of $\{1,2,\ldots,M\}$ with $n$ colors contains a monochromatic monotone path of length $t$.
    Fox et al~\cite{Fox2011ErdsSzekerestypeTF} established upper and lower bounds for $M_3(t,n)$ and suggested closing the gap between these bounds as an interesting open question. Moshkovitz and Shapira \cite{Moshkovitz2012RamseyTI} connected this Ramsey-type question with the antichains of $[t]^n$ by showing that
    \begin{align}\label{eq:MSid}
    M_3(t,n)=\alpha([t]^n)+1\end{align}
    and used it to obtain the bounds
    $$2^{\frac{2}{3}t^{n-1}/\sqrt n} \leq M_3(t,n) \leq 2^{2t^{n-1}}.$$
    Moshkovitz and Shapira further asked whether the relation
    \begin{align}\label{eq:MSqu}
    M_3(t,n)=2^{(1+o_n(1)){\ell(t,n)}}
    \end{align}
    is true as $n\rightarrow\infty$,
    where we recall that $\ell(t,n)$ is the size of a largest layer of the poset $[t]^n$.    
    Note that the lower bound in~\eqref{eq:MSqu} is clear from~\eqref{eq:MSid} by taking subsets of a largest layer in $[t]^n$. 
    The question of Moshkovitz-Shapira was positively answered by Falgas-Ravry et al \cite{falgasravry2023dedekinds}. We note that Anderson \cite{Anderson1967Primitive} proved that
    \begin{equation}\label{eqn:anderson}
            {\ell(t,n)}=(1+o_n(1))t^{n}\sqrt{\frac{6}{\pi(t^2-1)n}}.
    \end{equation}
    In the language of the current section, our results provide a precise estimate for $M_3(3,n)$ and $P_{n-1}(3)$. For example, we have the following immediate corollary to \Cref{thm:asymp_1}.
\begin{corollary}\label{thm:ramsey}
        With $\ell_n$, $T_1$, and $T_2$ as in \Cref{thm:asymp_1},
        $$M_3(3,n)=\left(1+e^{-\gO(n)}\right)2^{\ell_n}\exp{(T_1+T_2)}.$$
\end{corollary}

\subsection{Methods} In this section we highlight the main methods used in this paper. We refer the reader to \Cref{sec:prelim} for the terms not defined in this section.

\subsubsection{Graph containers and cluster expansion} 
Sapozhenko's graph container method, also pioneered in the work of Kleitman and Winston \cite{kleitman1982number}, has become a very powerful tool in the study of independent sets in graphs, e.g.
\cite{balogh2025maximal, balogh2021independent, galvin2006slow, kahn2022number, park2022note, potukuchi2021enumerating}, and beyond \cite{balogh2024intersecting, krueger2024lipschitz}.
The first author and Perkins \cite{hypercube} used the container method in conjunction with polymer models and the cluster expansion method from statistical physics, to obtain a detailed description of independent sets in the hypercube. This framework has since been used to answer many enumerative and algorithmic questions in combinatorics, e.g., \cite{balogh2021independent, collares2025counting, davies2021proof, geisler2025counting, JMP1, jenssen2022independent, jenssen2023approximately, jenssen2023evolution, kronenberg2022independent, li2023number}.
Our work is based in the same framework, and at a high level follows the blueprint of a recent paper of the first two authors and Malekshahian \cite{JMP1} on Dedekind's problem. Here the application of the container method is more challenging as we have to go beyond the usual condition in \eqref{Sap.cond}; to the best of our knowledge, up until now, the host graphs of all the other applications satisfy \eqref{Sap.cond}. The classical container families are parametrized by the sizes of the ``closure"  and the neighborhood of vertex subsets. On the other hand, the degree constraints and the irregularity of $[3]^n$ lead us to consider a new characterization of families of containers, which are additionally parametrized by the number of edges from the neighborhood which ``escape'' the closure. We note that our methods could be extended to obtain finer structural characterizations than Theorems \ref{thm.central} and \ref{thm:CLT}, and to count antichains of a fixed size, similar to the results in, e.g., \cite{JMP1, hypercube, jenssen2022independent}. We believe that a version of our main container lemma (\Cref{lem:container.graph}) should hold for the bipartite graphs induced by two adjacent layers of $[t]^n$ for all fixed $t$, despite the significant irregularity of these graphs. Identifying the minimal conditions under which a container lemma for irregular graphs holds would be of independent interest, as it would greatly expand the applicability of the method.

\subsubsection{Isoperimetric inequalities}
Vertex-expansion is the key structural property necessary for proving graph container lemmas in this paper. In the early works of Korshunov and Sapozhenko on the Boolean lattice and the hypercube \cite{korshunov, Sapozhenko1989}, the required isoperimetric estimates follow immediately from a characterization of sets with minimal expansion \cite{bezrukov1985minimization, katona, kruskal}. More precisely, Kruskal-Katona's Theorem states that over all sets of given size $m$ in a layer of the Boolean lattice, the minimal downwards (upwards, resp.) expansion is achieved by the first (last, resp.) $m$ elements of the layer ordered lexicographically. This extremal result can be easily translated to a quantitative estimate of the expansion. On the other hand, despite the existence of a Kruskal-Katona-type theorem for $[t]^n$ (in fact, any product of chains) by Clements and Lindstr{\"o}m \cite{clements1969generalization}, its translation to a quantitative result is far from immediate. For $[3]^n$, partial results for the expansion of individual layers were obtained in \cite{NSS}. Although our main theorems concern $[3]^n$, we provide isoperimetric estimates that hold for any $[t]^n$ with $t$ fixed. We believe they are of independent interest,  and may prove useful for generalizing our approach to $\alpha([t]^n)$. A key ingredient of our proof is an estimate for the expansion of layers via a detailed Local Limit Theorem due to Esseen \cite{Esseen_1945}. We provide a different, simpler combinatorial proof for $t=3$ in the main body of the paper, and defer the probabilistic arguments (that work for every fixed $t$) to the appendix.

\subsection{Organization} \Cref{sec:prelim} briefly provides the necessary framework and tools from probability and statistical physics that will be used. The isoperimetric inequalities for $[t]^n$ are proved in \Cref{sec:isoper}.
\Cref{sec:containers} discusses our container lemma for a class of irregular graphs of which two consecutive layers of $[3]^n$ are a special case, along with a necessary variant required for treating a poset with three layers. The rest of the paper uses these tools to prove our main theorems: in \Cref{sec:reduction} we show it suffices to consider only the three central layers (\Cref{thm.reduction}) and in \Cref{sec:central} we solve the counting problem in the three central layers (\Cref{thm.central}). Finally, \Cref{sec:CLT} is concerned with the structure of typical antichains in the central layers, proving the Central Limit Theorem for the defects (\Cref{thm:CLT}). \Cref{sec:computations} contains all the deferred computations of terms from the cluster expansion, in particular, \Cref{thm:asymp_2} will follow from the computations.

\section{Preliminaries}\label{sec:prelim}
\subsection{Tools from probability}\label{sec:prob}
For a random variable $X$, the \textit{cumulant generating function} of $X$ is the natural logarithm of its moment generating function, that is, $h_X(t)\coloneqq \ln \mathbb{E} [e^{tX}].$ The $\ell^{\text{th}}$ cumulant of $X$ is then 
\[ \kappa_{\ell}(X)\coloneqq \left. \frac{\partial^{\ell} h_X(t)}{\partial t^{\ell}}\right|_{t=0}. \]
The first two cumulants are $\kappa_1(X)=\mathbb{E}[X]$ and $\kappa_2(X)=\text{Var}(X)$. Note that for constants $a, b>0$ and $\ell\ge 2$ we have
\beq{cumlin}\kappa_\ell ((X-a)/b)=\kappa_\ell(X)/b^\ell. \enq
We will use the following well-known property of cumulants (see, e.g. \cite{janson1988normal}). 
\begin{fact}\label{cumulantfact}
    Let $(X_n)_{n\geq 1}$ be a sequence of real-valued random variables such that $\kappa_1(X_n)\rightarrow 0,$ $\kappa_2(X_n)\rightarrow 1$ and $\kappa_\ell(X_n)\rightarrow 0$ for all $\ell\geq 3$ as $n \rightarrow \infty$. Then the sequence $(X_n)_{n\geq 1}$ converges in distribution to a standard Gaussian.
\end{fact}
\nin We note that this is simply a restatement of the method of moments (see, e.g., \cite[Section 30]{billingsley_1995}) in the language of cumulants: if $M_X(t)$ is the moment generating function of $X$, the moments can be recovered from the equation $M_X(t)=e^{h_X(t)}$ in the ring of formal power series. From this, it follows that each moment is determined by a finite number of cumulants, which leads to \Cref{cumulantfact}.

\subsection{Polymer models and cluster expansion} \label{sec:polymer}

Let $\cP$ be a finite set, called the set of \textit{polymers}. Let $w: \cP\rightarrow \mathbb{C}$ a weight function and $\sim$ a symmetric and antireflexive relation on $\cP$ which we refer to as the \textit{compatibility} relation. We also say two polymers $\gamma$ and $\gamma'$ are incompatible if $\gamma \nsim \gamma'$. We refer to the triple $(\mathcal{P}, \sim, w)$ as a \textit{polymer model}.

A set of pairwise compatible polymers is called a \textit{polymer configuration}. Let $\Omega=\Omega(\mathcal{P},\sim)$ be the collection of all polymer configurations, including the empty set of polymers. We define the \textit{partition function} $\Xi=\Xi(\cP, \sim, w)$ as
\begin{equation*}
    \Xi= \sum\limits_{\Lam \in \Omega} \prod\limits_{\gamma \in \Lam} w(\gamma),
\end{equation*}
where by convention, we take the contribution from the empty set to be $1$. 

Let $\Gamma=(\gamma_1, \ldots, \gamma_k)$ be a sequence of polymers that are not necessarily distinct. The \emph{incompatibility graph} of $\Gamma$ is the graph $G_\Gamma$ on vertex set $[k]$ where $i$ is adjacent to $j$ if and only if $\gamma_i\nsim \gamma_j$.
We call $\Gamma$ a \emph{cluster} if $G_\Gamma$ is connected.
We define the \emph{size} of $\Gamma$ to be $\|\Gamma\|\coloneqq \sum_{i=1}^k |\gamma_i|$. We define the weight of a cluster $\Gamma$ to be 
\begin{align}\label{eq:wclusterdef}
    w(\Gamma)=\phi(G_\Gamma) \prod\limits_{{\gamma} \in \Gamma} w(\gamma),
\end{align}
where $\phi$ is the Ursell function of a graph $G=(V,E)$, defined as
\beq{eq:Ursell}
    \phi(G):=\frac{1}{|V|!} \sum\limits_{\substack{E' \subseteq E: \\ (V, E') \text{ connected}}} (-1)^{|E'|}.
\enq
Let $\mathcal{C}=\mathcal{C}(\mathcal{P},\sim)$ be the set of all clusters. Then the \emph{cluster expansion} of $\ln \Xi$ is the formal power series
\begin{equation}\label{eqclusterexp}
   \ln \Xi =\sum\limits_{\Gamma \in \mathcal{C}} w(\Gamma).
\end{equation}
This turns out to be the multivariate Taylor series of $\ln \Xi$ - see \cite{dobrushin, ScoS05}. In order to use the cluster expansion computationally, it is essential to verify its convergence. A powerful tool that gives a sufficient condition for convergence, as well as tail bounds, is the so-called \textit{Koteck\'y-Preiss condition}. For any function $g:\cP\rightarrow[0,\infty)$ and cluster $\Gam\in \cC$, define $g(\Gam) := \sum_{\gam\in\Gam}g(\gam)$. We say that a polymer $\gam$ is incompatible with a cluster $\Gam$ if there exists some $\gam'\in \Gam$ such that $\gam\nsim \gam'$, and we denote this by $\Gam\nsim A$.

\begin{theorem} [Koteck\'y and Preiss \cite{kp}]\label{thm.KP}
Let $f,$ $g: \cP \rightarrow [0, \infty)$ be two functions. If, for all $\gam \in \cP$,
\begin{equation}\label{KPbound1}
    \sum\limits_{{\gam'} \nsim {\gam}} |w({\gam}')|e^{f({\gam}')+g({\gam}')} \leq f({\gam}),
\end{equation}
then the cluster expansion~\eqref{eqclusterexp} converges absolutely. Furthermore, for all polymers ${\gamma}$,
\beq{kpbounds}    \sum\limits_{\substack{\Gam \in \cC \\ \Gam \nsim {\gam}}} |w(\Gam)|e^{g(\Gam)}\leq f({\gam}). \enq
\end{theorem}

\section{Isoperimetry}\label{sec:isoper}

\Cref{sec:isoper.tn} collects a few classical isoperimetric results for $[t]^n$. In \Cref{sec:isoper.3n}, we provide some additional isoperimetric inequalities for $[t]^n$ that will be used for the proof of our main results. As mentioned earlier, we will only need the results for $t=3$, but keep the statements general in case they prove useful elsewhere. First, we introduce some additional notation.

Recall that $L_k=L_k(t,n)$ denotes the $k$-th layer of $[t]^n$, and we will often suppress the dependence on $t$ or $n$ (or both) if the context is clear. For $k\in \{0,\ldots,(t-1)n-1\}$ and $v\in L_k$, we write $N^+(v)$ for the set of neighbors (in the Hasse diagram of $[t]^n$) of $v$ in $[t]^n$ that belong to the layer $L_{k+1}$, and $N^+(S):=\cup_{v \in S} N^+(v)$ for any $S \sub [t]^n$. For $i\in \{1, \ldots, (t-1)n\}$ and $v\in L_k$ we similarly define $N^-(v)$ and $N^-(S)$. Finally, we denote by $S^j$ $(j \in \{0, \ldots, t-1\}$) the set of elements of $S$ whose first coordinate is $j$.

\subsection{Classical results for $[t]^n$}\label{sec:isoper.tn}

    \begin{lemma}[Anderson \cite{anderson1968divisors}]\label{lem:logconcave} The sizes $\ell_j$ of the layers $L_j$ of $[t]^n$ form a log-concave sequence. That is, for any $j \in [(t-1)n-1]$,
    \[\ell_j^2 \ge \ell_{j-1}\ell_{j+1}.\]
    \end{lemma}
Recall that a bipartite graph with parts $X, Y$ has the \textit{normalized matching property} if
    \beq{eq:nmp}\frac{|N(A)|}{|A|}\geq \frac{|Y|}{|X|} ~\text{ for every $A\subseteq X$}.\enq
(Note that the choice of $X$ or $Y$ is not important in the above definition, as the above implies $|N(B)|/|B| \ge |X|/|Y|$ for every $B \sub Y$.)

A graded poset is said to have the normalized matching property if the bipartite graph induced by the comparability graph on any two distinct layers has the normalized matching property.
    
    \begin{lemma} [Anderson \cite{anderson1968divisors} and Harper \cite{harper1974morphology}]
        The poset $[t]^n$ has the normalized matching property.
    \end{lemma}

    Recall that the lexicographic order, $\lex$, on $[t]^n$ is defined as follows: $a\lex b$ iff there is an $i$ such that $a_j=b_j$ for all $j<i$ and $a_i<b_i$. 

    \begin{definition}\label{def:compression}
    Let $S\subseteq L_k$ for some $k$. The \emph{compression} of $S$, denoted by $C(S)$, is the set containing the first $|S|$ elements of $L_k$ according to the lexicographic order. Denote by $L(S)$ the last $|S|$ elements of $L_k$ in the lexicographic order.
    \end{definition}
    The next theorem is an analogue of the Kruskal-Katona Theorem \cite{katona, kruskal} for $[t]^n$.
\begin{proposition}[Clements-Lindstr\"om \cite{clements1969generalization}]\label{prop:CL}
            For any layer $L_k$ and any $S\subseteq L_k$
            $$N^-(C(S))\subseteq C(N^-(S)) \text{   and   } N^+(L(S))\subseteq L(N^+(S)).$$
            Consequently, $|N^-(S)|\geq |N^-(C(S))|$ and $|N^+(S)|\geq |N^+(L(S))|$.
    \end{proposition}

\subsection{Isoperimetric inequalities for $[t]^n$}\label{sec:isoper.3n}

A key lemma for our proof is an estimate on the ratio between the sizes of two consecutive layers in $[t]^n$ (\Cref{lem:layer_exp}). The proof for general $t$ uses probabilistic tools, and we present it in \Cref{sec:app_isop}. For the special case of $t=3$, which is enough for the purpose of the present paper, the statement also follows from a (simpler) combinatorial argument, and we present it here.

A set of pairwise comparable elements of a poset is called a \textit{chain}. A chain $C=\{x_1<\dots<x_r\}$ is \textit{saturated} if $x_{i+1}$ covers $x_i$ for every $i$. Let $P$ be a graded poset with $m$ layers. A \textit{symmetric chain decomposition} (SCD) of $P$ is a partition of the elements of $P$ into saturated chains $C_1,...,C_M$, such that for all $1\leq s\leq M$, if $C_s=\{x_1<\dots<x_r\}$ with $x_1\in L_i$ and $x_r\in L_j$, then $i+j=m$.

We use the SCD construction for $[t]^n$ due to Tsai \cite{Tsai2019}, which extends Greene-Kleitman's SCD construction for $\{0,1\}^n$ \cite{Kleitman1976OnJCP}. For each $x=(x_1,\ldots, x_n) \in [t]^n=\{0,1,\ldots,t-1\}^n$, we define the corresponding "bracket configuration" as follows: consider $n$ consecutive blocks where each block consists of $t-1$ blanks. We fill in the blanks in the $i$th block from the left, first with $x_i$ right brackets ')' followed by $t-x_i-1$ left brackets '(' (see \Cref{fig:bracket}).

\begin{figure}[h!]
    \centering
    \includegraphics[width=9cm]{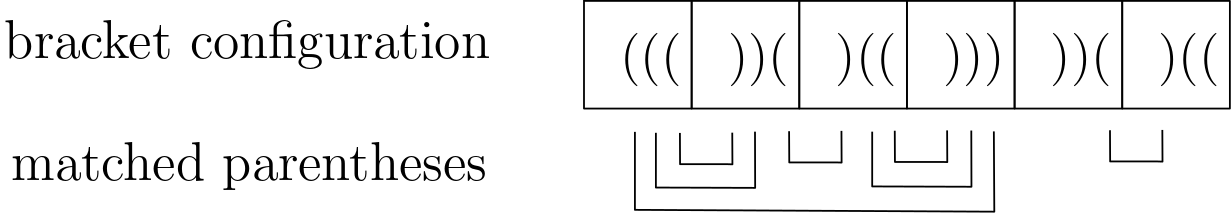}
    \caption{The bracket configuration of $x=(0,2,1,3,2,1) \in [4]^6$. Its bracket structure $B(x)$ is $(((\ | \ ))(\ | \ )((\ | \ )))\ | \ **(\ | \ )**$.}
    \label{fig:bracket}
\end{figure}

Given a bracket configuration, we match brackets in the natural way: scanning from left to right, we match a right bracket with the closest unmatched left bracket to its left. The ``bracket structure'' $B(x)$ of $x$ is defined by taking the bracket configuration of $x$ and replacing unmatched brackets with the symbol `$\ast$'. For example, if $x=(0,2,1,3,2,1)\in [4]^6$, then the corresponding bracket configuration and bracket structure are shown in \Cref{fig:bracket}. 

For each $x \in [t]^n$, define the equivalence class $C_x=\{y \in [t]^n:B(y)=B(x)$\}. For example, for the above $x$, $C_x=\{(0,2,1,3,0,1), (0,2,1,3,1,1), (0,2,1,3,2,1), (0,2,1,3,2,2), (0,2,1,3,2,3)\}$. {Tsai \cite{Tsai2019} showed that $\cC:=\{C_x:x \in [t]^n\}$ is a SCD of $[t]^n$. Observe that, if $\mathcal C$ is a SCD of $[t]^n$, then $|\mathcal C|=\ell_m(t,n)$ where $m:=\lfloor\frac{(t-1)n}{2}\rfloor$ (so $L_m(t,n)$ is a middle layer of $[t]^n$),} since each point in a middle layer is contained in exactly one chain in $\cC$.
\begin{lemma}\label{lem:layer_exp} 
There exists $c>0$ for which the following holds: for any integer $t \ge 2$ there exists $n_t>0$ such that for all $n\geq n_t$ and $j \in [1,m]$,
    $$\frac{\ell_j(t,n)}{\ell_{j-1}(t,n)}\geq 1+\frac{c}{t^2n}.$$
\end{lemma}

As mentioned, here we only prove the statement for $t=3$, deferring the proof for general $t$ to \Cref{sec:app_isop}.

\begin{proof}  For $t=3$, we have $m=n$, and since $\frac{\ell_{j+1}}{\ell_j}\le \frac{\ell_j}{\ell_{j-1}}$ for any $j \ge 0$ by the log-concavity in \Cref{lem:logconcave}, it is enough to prove the lemma for $j=n$.

We will use Tsai's SCD $\cC$ that we just described. Specifically when $t=3$, for each $x=(x_i)\in [3]^n= \{0, 1,2\}^n$, we have $n$ blocks and we place `$(($', `$)($', `$))$' in block $i$ if $x_i=0,1,2$ respectively. Furthermore, recall that $|\cC|=\ell_n$.

This chain decomposition naturally defines an injection $\varphi: L_{n-1}\to L_n$ where $\varphi(x)$ is the unique $y\in L_n$ such that $y\in C_x$. To lower bound the ratio $\ell_n/\ell_{n-1}$ we will lower bound the number of elements $x\in L_n$ \emph{not} in the image of $\varphi$. These are precisely the elements $x\in L_n$ such that $C_x$ is a singleton which certainly occurs if 
\[\text{all of the brackets in the bracket configuration of $x$ are matched}\]
(i.e. $B(x)$ contains no `$\ast$' symbol). Let us call such an $x$ `fully matched'. To count the number of fully matched $x\in L_n$ we note that they naturally lie in bijection with a set of so-called `Motzkin paths'.

A Motzkin path of length $n$ is a path in the plane integer lattice $\mathbb{Z}^2$ starting at $(0,0)$ and terminating at $(n,0)$ consisting of up steps $(1, 1)$, down steps $(1, -1)$ and horizontal steps $(1,0)$, which never passes below the $\mathrm{x}$-axis. We let $M_n$ denote the number of Motzkin paths of length $n$ (also known as the $n$th Motzkin number). 
By associating `$(($' with an up step, `$)($' with a horizontal step, and `$))$' with a down step, the set of fully matched $x\in L_n$ is seen to be in bijection with the following set of Motzkin paths of length $n$: those that always take an up step after hitting the  $\mathrm{x}$-axis (including the first step). Indeed, if we think of writing the bracket configuration of  $x \in [3]^n$ one step at a time, then every time we hit the $\mathrm{x}$-axis, all parentheses are matched. Then, if the next step is horizontal, the left parentheses of $)($ will not be matched. The number of such paths is clearly lower bounded by the number of Motzkin paths that travel from $(1,1)$ to $(n-1,1)$ without ever passing below the line $y=1$. The number of paths of this latter type is clearly $M_{n-2}$. We conclude that the number of fully matched $x\in L_n$ is at least $M_{n-2}$ and so 
\[
\ell_n\geq \ell_{n-1}+ M_{n-2}.
\]
The asymptotics of $M_n$ are well-known (see e.g.~\cite{Anderson1967Primitive}) in particular
\[
M_{n-2}=\Theta(3^n/n^{3/2})\, .
\]
The result follows by recalling that $\ell_{n-1}\leq \ell_n=\ell(3,n)\stackrel{\eqref{eqn:anderson}}{=}\Theta(3^n/n^{1/2})$.
\end{proof}

\begin{proposition}\label{prop:isoperim_up}
        Let $c, n_t$ be as in \Cref{lem:layer_exp}. For any $i< m=\lfloor\frac{(t-1)n}{2}\rfloor$, the following holds.  If $S\subseteq L_i(t,n)$, then for all $n\geq n_t$,
        \begin{enumerate}[(a)]
            \item $|N^+(S)|\geq \frac{1}{2}|S|(n-|S|)$;
            \item $|N^+(S)|\geq \left(1+\frac{c}{t^2n}\right)|S|$;
            \item there is a constant $c'>0$ such that if  $|S|\leq n^4,$ then $|N^+(S)|\geq c'n|S|/t$.
        \end{enumerate}
    \end{proposition}
\begin{proof}
(a) Note that for any $x \in [t]^n$, $|N^+(x)|=n-|\{i: x_i=t-1\}|$. So if $i<m$, then any $x \in L_i$ satisfies $|N^+(x)|> n/2$. Also noting that $|N^+(x) \cap N^+(y)| \le 1$ for any distinct $x,y \in L_i$,
\[|N^+(S)|\ge \sum_{x \in S} |N^+(x)|-\sum_{\substack{x,y \in S \\ x \ne y}} |N^+(x) \cap N^+(y)|\ge \frac{n}{2}|S|-{|S| \choose 2}\ge \frac{1}{2}(n|S|-|S|^2).\]

\nin (b) By the normalized matching property \eqref{eq:nmp} and \Cref{lem:layer_exp}, ${|N^+(S)|}/{|S|}\ge {|L_{i+1}|}/{|L_i|} \ge 1+\frac{c}{t^2 n}.$

\nin (c)  By \Cref{prop:CL}, we may assume that $S$ comprises of the last $|S|$ elements of $L_i=L_i(t,n)$ in the lexicographic order. Let $k=\max\{\lfloor \frac{i}{t-1}\rfloor-5,0\}~(\le n/2)$ and note that 
\[
S'=\{x\in L_i : x_j=t-1 \text{ for all } j\leq k \}
\]
is a final segment of the lexicographic ordering on $L_i$. We will show that $|S'|\geq |S|$ so that $S\subseteq S'$. If $k=0$, then $S'=L_i$ and so $|S'|\geq |S|$ trivially. Assume then that $k>0$ i.e. $k=\lfloor \frac{i}{t-1}\rfloor-5$, and note that
\begin{align}\label{eq:ktineq}
i-6(t-1)\leq k(t-1)\leq i-5(t-1)\, .
\end{align}
Let $S''$ denote the set of $x\in S'$ for which $x_j=1$ for exactly $i-k(t-1)~(>0)$ values of $j>k$ and note that
\[
|S'|\geq |S''|= \binom{n-k}{i-k(t-1)}\geq \binom{n/2}{5(t-1)}\geq n^4 \geq |S|
\]
where for the second inequality we used that $k\leq n/2$ and $i-k(t-1)\geq 5(t-1)$ by~\eqref{eq:ktineq}. 

Suppose now that $x\in S$. Since $S\subseteq S'$, the first $k$ coordinates of $x$ are equal to $t-1$. In particular,
\[
i=\sum_{j\geq 1} x_j = k(t-1) + \sum_{j>k} x_j\, .
\]
Since $i-k(t-1)\leq 6(t-1)$ by~\eqref{eq:ktineq}, we conclude that $\sum_{j>k} x_j \leq 6(t-1)$. In particular, $x$ has at least $n-k-6\geq n/3$ neighbours in $L_{i+1}$ since $x_j=t-1$ for at most $6$ values of $j>k$. On the other hand, if $y$ is such a neighbour, then $\sum_{j>k} y_j \leq 6(t-1)+1<6t$. In particular, $y_j>0$ for at most $6t$ values of $j>k$, therefore 
$y$ has at most $6t$ neighbours in $S$. It follows that
\[
\frac{|N^{+}(S)|}{|S|}\geq \frac{n/3}{6t}
\]
as desired. 
\end{proof}

\section{Graph containers}\label{sec:containers}

In this section, we prove a graph container lemma for bipartite graphs with certain local degree conditions. We first introduce 
some basic definitions and notation. For a graph $H$, $E(H)$ and $V(H)$ denote the set of edges and vertices of $H$ respectively. For $u, v \in V(H)$, $\dist_H(u,v)$ (or simply $\dist(u,v)$ if the host graph is clear) denotes the length of a shortest path in $H$ between $u$ and $v$. For an integer $k$, $H^k$ denotes the $k$-th power of $H$, that is, $V(H^k)=V(H)$ and $u, v \in H^k$ are adjacent iff $u \ne v$ and $\dist_H(u,v) \le k$. For $U \sub V(H)$, $H[U]$ denotes the induced subgraph of $H$ on $U$.  We say $A \sub V(H)$ is \textit{$k$-linked} if $H^k[A]$ is connected. For $A,B \sub V(H)$, $\nabla(A,B):=\{\{a,b\} \in E(H):a \in A, b \in B\}$ and $\ov{A}:=V(H) \setminus A$. The \textit{closure} of $A$, denoted by $[A]$, is $[A]:=\{v\in V(H):N(v)\subseteq N(A)\}$.

Throughout the section, we assume that $\Sigma$ is a bipartite graph on the bipartition $X \cup Y$.  
 Write $d(v)$ or $d_v$ for the degree of $v$ in $\Sigma$. Let $\gd \in (0,1]$ and $d$ and $\Delta$ be integers. In our applications, $d$ is usually large, and asymptotic notation in this section holds as $d \rightarrow \infty$ where implicit constants (only) depend on $\delta$ and $\Delta$. We assume that $\Sigma$ satisfies the following properties.
\beq{eq:deg} d(v) \in [\gd d, d] \quad \forall v \in X \cup Y;\enq
\beq{eq:dv.vs.dw} d(v)\ge d(w) \quad \text{if } v \in X, w \in Y \text{ and } v \sim w;\enq
\[|N(v) \cap N(w)| \le \Delta \quad \text{if } v \ne w.\]
Notice that in $[3]^n$, both $L_{n-2} \cup L_{n-1}$ and $L_{n-1} \cup L_{n}$\footnote{ For a set $S\subseteq [3]^n$, we abuse notation somewhat and identify $S$ with the subgraph of the Hasse diagram of $[3]^n$ induced by the set $S$.} satisfy the above properties for $\gd=1/2-o(1)$, $d=n$ and $\Delta=1$.

Given a graph $\Sigma$ as above, $v \in X$ and integers $a$ and $g$, set
\[\cG(a,g):=\{A \sub X: A \text{ is $2$-linked}, |[A]|=a, |N(A)|=g\}.\]
The main contribution of this section is the following version of the graph container lemma for the class of graphs described above. 

\begin{lemma}\label{lem:container.graph}
There exists a constant $\gamma>0$ for which the following holds. If $g-a\gg \frac{ g\log^3d}{d^2}$ and $g-a \gg \log^2 d$, then
\[|\cG(a,g)|\le |Y|\cdot2^{g-\gamma(g-a)/\log d}.\]   
\end{lemma}

The crucial difference between this result and previous results of a similar kind, is the lack of assumption \eqref{Sap.cond} in our setting. To handle this obstacle, we introduce a key parameter $\kappa(A):=|\nabla(N(A), \ov{[A]})|$ and partition $\cG(a,g)$ according to the value of $\kappa$. Given an integer $\gk$, define a subfamily $\cG(a,g,\gk)$ of $\cG(a,g)$ as follows.
\[\cG(a,g,\gk):=\{A \sub X: A \text{ is $2$-linked}, |[A]|=a, |N(A)|=g, |\nabla(N(A), \ov{[A]})|=\gk\}.\]
In order to obtain \Cref{lem:container.graph}, we first prove the following lemma. 

\begin{lemma}\label{lem:container.graph'} There exists a constant $\gamma'>0$ for which the following holds. If $\kappa> 0$ and $g-a\gg \frac{g\log^3d}{d^2}$, then
\[|\cG(a,g, \gk)|\le|Y| \cdot 2^{g-\gamma'(g-a)/\log d}.\]    
\end{lemma}

To show that \Cref{lem:container.graph} follows from \Cref{lem:container.graph'}, we first show that we can control $\kappa$ using $t:=|N(A)|-|[A]|$ by the degree assumptions on $\Sigma$. (Notice that now  parameter $t$ has a different meaning from the earlier sections; the choice of the letter $t$ is in accordance with related works, e.g. \cite{Galvin2019, jenssen2024refined, JMP1}.)

\begin{proposition}\label{prop:t.kappa} For any $A \sub X$,
\beq{eq:t.kappa} |\nabla(N(A), \ov{[A]})|\le d(|N(A)|-|[A]|).\enq
\end{proposition}

\begin{proof} We use $a, g$, and $t$ for $|[A]|, |N(A)|$, and $|N(A)|-|[A]|$, respectively. Then
    \[\begin{split}(g-t=)~a=\sum_{x \in [A]}\sum_{y \sim x}\frac{1}{d_x}\stackrel{\eqref{eq:dv.vs.dw}}{\le} \sum_{x \in [A]}\sum_{y \sim x} \frac{1}{d_y}=\sum_{y \in N(A)}\sum_{\substack{x \sim y \\ x \in [A]}}\frac{1}{d_y}&=\sum_{y \in N(A)}\sum_{x \sim y}\frac{1}{d_y}-\sum_{y \in N(A)}\sum_{\substack{x \sim y \\ x \notin [A]}}\frac{1}{d_y}\\
    &\stackrel{\eqref{eq:deg}}{\le} g-\frac{1}{d}|\nabla(N(A), \ov{[A]})|=g-\gk/d.\qedhere\end{split}\]
\end{proof}

\begin{proof}[Derivation of \Cref{lem:container.graph} from \Cref{lem:container.graph'}]
\Cref{prop:t.kappa} gives $\gk \le td$, so by \Cref{lem:container.graph'},
    \[|\cG(a,g)|=\sum_{\gk\le td}|\cG(a,g,\gk)|\le \sum_{\kappa \le td}|Y|\cdot 2^{g-\gamma' t/\log d} \le  |Y|\cdot td\cdot 2^{g-\gamma't/\log d}.\]
    The last expression is at most $|Y|\cdot2^{g-\gamma t/\log d}$ for a suitable constant $\gamma$ since $t\gg \log ^2 d$ by assumption.    
\end{proof}

We will prove \Cref{lem:container.graph'} in \Cref{subsec:lem4.2}. For the rest of the paper, we use $G=G(A)$ for $N(A)$ following the convention in the literature, often suppressing its dependence on $A$.

\subsection{Preliminaries for the proof of \Cref{lem:container.graph'}} Write $d_Z(v)$ for $|N(v) \cap Z|$ for $Z \sub V(\Sigma)$.

\begin{definition}
        Given $A \sub X$, let $G^\phi=\{v\in G:d_{[A]}(v)>\phi\}$. A $\phi$-approximation for $A$ is an $F'\subseteq Y$ such that
        \begin{equation}\label{eqn:phi_1}            G^\phi\subseteq F'\subseteq G; \text{ and}
        \end{equation}        \begin{equation}\label{eqn:phi_2}
            N(F')\supseteq[A].
        \end{equation}
    \end{definition}

    \begin{lemma}\label{lem:phi} For any $\phi \in [1, \gd d-1]$, there is a family $\cV=\cV(a,g,\gk,\phi)$ with 
\[|\cV|\le |Y|\exp\left[O_{\delta, \Delta}\left(\frac{g\log^2 d}{\phi d}+\frac{\gk\log^2d}{d(\gd d-\phi)}+\frac{\gk\log^2 d}{\phi d}\right)\right]\]
        such that each $A\in\cG(a,g,\gk)$ has a $\phi$-approximation in $\cV$.
    \end{lemma}

The proof of~\Cref{lem:phi} is almost identical to that for \cite[Lemma 5.4]{Galvin2019}; the only difference is that the role of $t$ in the proof of \cite[Lemma 5.4]{Galvin2019} is replaced with $\kappa$. We defer the details to \Cref{app:lem:phi}.

        \begin{definition}
        Given $A \sub X$, a $\psi$-approximation for $A$ is a pair $(S,F)\subseteq 2^X\times 2^Y$ satisfying 
        \begin{equation}\label{eqn:psi_1}
            F\subseteq G,\ S\supseteq [A];
        \end{equation}
        \begin{equation}\label{eqn:psi_2}
            d_F(u)\geq d(u)-\psi ~\ \forall u\in S; \text{ and}
        \end{equation}
        \begin{equation}\label{eqn:psi_3}
            d_{X\setminus S}(v)\geq d(v)-\psi ~\ \forall v \in Y\setminus F.
        \end{equation}
    \end{definition}

        \begin{lemma}\label{lem:psi} Given $\psi \in [1, \gd d-1]$ and $F' \sub Y$, write $\cG(a,g,\gk,F')$ for the collection of $A \in \cG(a,g,\gk)$ for which $F'$ is a $\varphi$-approximation. Then there exists a family $\cW=\cW(a, g, \gk, F') \sub 2^Y \times 2^X$ with
        \beq{psi_bd} |\cW|\leq {dg \choose \le \gk/((\gd d-\varphi)\psi)}{ d^2 g \choose \le \gk/((\delta d-\psi)\psi)} \enq
        such that every $A\in\cG(a,g,\gk, F')$ has a $\psi$-approximation in $\cW$.
    \end{lemma}

We again defer the proof to \Cref{app:lem:psi}, as it is almost identical to that for \cite[Lemma 5.5]{Galvin2019} with the role of parameter $t$ replaced with $\gk$.

    \begin{proposition}\label{prop:SF.facts} Let $(S,F)$ be a $\psi$-approximation for some $A\in\cG(a,g, \gk)$, and suppose $\psi \ll d$. Then,
    \begin{enumerate}[(a)]
    \item $|S \setminus [A]|=O(\gk/d)$ and $|G \setminus F|=O(\gk/d)$;
\item $|\nabla (S, \ov F)| =O\left(\gk \psi/d\right)$;
\item $|S|\le |F|+O(\gk\psi/d^2)$.
    \end{enumerate}
\end{proposition}

\begin{proof} Write $s$ and $f$ for $|S|$ and $|F|$, respectively.

\nin (a) Note that
    $(s-a)(\gd d-\psi)\le |\nabla(S\setminus[A],F)| \le  |\nabla(G, \ov{[A]})|=\gk,$ so $s-a\le \gk/(\gd d-\psi)=O(\gk/d)$. Similarly, $(g-f)(\gd d-\psi) \le |\nabla(G \setminus F, \ov S)| \le |\nabla (G, \ov{[A]})|=\gk$.

\nin (b) We have $\nabla(S,F)- \nabla([A],G)=\nabla(S \setminus [A],F)-\nabla([A], S \setminus F)$, so
   \[\nabla(F)-\nabla(F,\ov S)-\nabla([A])=\nabla(S \setminus [A])-\nabla(S \setminus [A],\ov F)-\nabla([A], G \setminus F).\]
   Using $|\nabla(S \setminus [A],\ov F)| \le \psi(s-a)$ and $|\nabla([A], G \setminus F)| \le \psi(g-f)$,
   \[|\nabla(F)|-|\nabla(F, \ov S)| \ge |\nabla(S)|-\psi(s-a+g-f),\]
   from which
   \[|\nabla(S, \ov F)|=|\nabla(S)|-|\nabla(S,F)| \le |\nabla(F)|-|\nabla(F,\ov S)|-|\nabla(S,F)|+\psi(s-a+g-f) \le O\left(\frac{\gk\psi}{d}\right)\]
   where we use item (a) for the last inequality.

\nin (c) By item (b), it suffices to show that $|S|\le|F|+|\nabla(S,\ov{F})|/(\gd d).$
For $x \in S$, set $d'_x=|N(x) \setminus F|$. Write
\[s=\sum_{x \in S}\frac{d_x-d_x'}{d_x}+\sum_{x \in S} \frac{d_x'}{d_x}.\]
The first term of the rhs is
\[\sum_{x \in S}\sum_{y \in F, y \sim x}\frac{1}{d_x} \stackrel{\eqref{eq:dv.vs.dw}}{\le} \sum_{x \in S}\sum_{y \in F, y \sim x}\frac{1}{d_y} = \sum_{y \in F} \sum_{x \in S, x \sim y} \frac{1}{d_y}\le f; \]
and the second term of the rhs is at most $|\nabla(S, \ov{F})|/(\gd d)$.
\end{proof}

\subsection{Proof of \Cref{lem:container.graph'}}\label{subsec:lem4.2} In this section and the next, `cost' means the number of choices for the objects under consideration. Recall the elementary facts that
    \beq{incr} \text{the function $x\log(1/x)$ is monotone increasing on $(0, 1/e)$ {and has a global maximum at $1/e$}; and}\enq
    \beq{binom} {n \choose \le k} \le \exp\left\{k\log\left(\frac{en}{k}\right)\right\} \quad \text{for $k \le n$.}\enq

     Set $\phi=\gd d/2$ and $\psi=cd/\log d$ for a small constant $c=c(\delta, \Delta)$ to be chosen later. In the asymptotic notation below, the implicit constants depend only on $\delta$ and $\Delta$, and we will suppress the dependence. (They are some absolute constants in our applications.) The cost for the $\psi$-approximations $(F, S)$ for $A$'s in $\cG(a, g, \gk)$ produced via \Cref{lem:phi} and \Cref{lem:psi} is at most (using \eqref{binom} to simplify the cost in \Cref{lem:psi})
\beq{container.cost} |Y|\exp\left[O\left(\frac{g\log^2 d}{d^2}+\frac{\gk\log^2 d}{d^2}+\frac{\gk \log d}{cd^2}\log\left(\frac{dg}{\gk}\right)\right)\right].\enq
In the next paragraph, we show that \eqref{container.cost} is bounded above by $|Y|\exp(o(t/\log d))$.

The first two terms in the exponent in \eqref{container.cost} can be handled immediately. By the assumption that $t\gg \frac{g\log^3d}{d^2}$, the first term, $g\log^2 d/d^2$, is bounded above by $o(t/\log d)$; by \Cref{prop:t.kappa} (that is, $\kappa\leq td$), the second term, $\kappa \log^2 d/d^2$, is bounded above by $t\log^2 d/d =o( t/\log d)$. For the last term, we consider two cases.
If $t/g\leq 1/e$, then
\begin{align*}
    \frac{\gk\log d}{cd^2}\log\left(\frac{dg}{\gk}\right)=\frac{g\log d}{ cd}\cdot \frac{\gk}{dg}\log\left(\frac{dg}{\gk}\right)\stackrel{\eqref{incr}}{\leq}\frac{t\log d}{cd}\cdot \log\left(\frac{g}{t}\right)\leq \frac{t\log d}{ cd}\log\left(\frac{d^2}{\log^3d}\right) =O\left(\frac{t\log^2d}{cd}\right),
\end{align*}
where the last inequality uses $t\gg g\log^3 d/d^2$.
On the other hand, if $1/e <  t/g~(<1)$, then we use the global maximum at $1/e$ to get
\begin{align*}
    \frac{\gk\log d}{cd^2}\log\left(\frac{dg}{\gk}\right)=\frac{g\log d}{ cd}\cdot \frac{\gk}{dg}\log\left(\frac{dg}{\gk}\right)\stackrel{\eqref{incr}} {\leq}O\left(\frac{g\log d}{cd}\right)
    =O\left( \frac{t\log d}{ cd}\right).
\end{align*}
So in any case, the last term is at most $O(\frac{t\log^2 d}{cd})=o(t/\log d)$.
In sum, \eqref{container.cost} is at most
\beq{container.cost'} |Y|\cdot e^{o(t/\log d)}.\enq
Next we claim that, given a pair $(S,F)$, the number of $A$'s in $\cG(a,g,\gk)$ which has  $(S,F)$ as a $\psi$-approximation is at most $2^{g-\gO(t/\log d)}$. This claim, combined with \eqref{container.cost'}, yields the conclusion of \Cref{lem:container.graph'}.

To prove the claim, we consider the following two cases. Fix a small constant $\gd'$, say, $\gd'=1/100$.

\nin \textit{Case 1.} $g-f \le \gd' t/\log d$.

In this case, we first specify $G \setminus F$ as a subset of $N(S)$. Since $S \sub N(F) \sub N(G)$, we have $|N(S)|\le d^2g$. Therefore, the cost for $G \setminus F$ is at most
\[{d^2g \choose \le \gd' t/\log d}\stackrel{\eqref{binom}}{\le}\exp\left\{\frac{\gd' t}{\log d}\log\left(\frac{ed^2g\log d}{\gd' t}\right)\right\}\le \exp\left\{\frac{\gd' t}{\log d}\log\left(\frac{ed^4}{ \gd'\log^2 d}\right)\right\}\overset{(\dagger)}{\le} e^{10\gd' t}<2^{t/2},\]
where $(\dagger)$ holds for large enough $d$,
and the last inequality is true for our choice of $\gd'$.

Once we specify $G \setminus F$, this determines $G$, thus $[A]$. We then specify $A$ as a subset of $[A]$, which costs $2^a=2^{g-t}$. Therefore, the total specification cost is at most
\[2^{g-t+t/2}\le2^{g-t/2}.\]
\nin \textit{Case 2.} $g-f \ge \gd' t/\log d$

In this case, we specify $A$ as a subset of $S$. Note that, by \Cref{prop:SF.facts} (c) and our choice of $\psi$,
\beq{eq:sg}s\le f+O\left(\frac{c\gk}{d\log d}\right) \stackrel{\eqref{eq:t.kappa}}{\le} f+O\left(\frac{ct}{\log d}\right) \le g-\frac{\gd' t}{\log d}+O\left(\frac{ct}{\log d}\right)\enq
for large enough $d$. Taking $c=c(\delta, \Delta)$ (recall that the $O(\cdot)$ above depends on $\delta$ and $\Delta$) small enough implies the specification cost is at most
\[2^{g-\gd' t/(2\log d)}.\]

\subsection{Graph containers for three layers} 

In this section we prove a graph container lemma for three consecutive layers in $[3]^n$ (\Cref{lem:containers})
 For $i\in [n]$ and a set $A\subseteq L_i=L_i(3,n)$, we slightly abuse notation and denote the \emph{upwards closure} of $A$ by $[A]\coloneqq \{ v\in L_i: N^+(v)\subseteq N^+(A)\}$. 
Recall that we say $A\subseteq [3]^n$ is $2$-linked if it is $2$-linked in the Hasse diagram of $[3]^n$.
Set
\begin{center}
$\cH(a,b,g,h):=$ \\ $\{(A,B)\in 2^{L_{n-1}} \times 2^{L_{n-2}}: A \text{ $2$-linked}, |[A]|=a, |[B]|=b, |N^+(A)|=g, |N^+(B)|=h, N^+(B) \subseteq A\}$
\end{center}
and let $t=g-a$ and $t'=h-b$.
\begin{lemma}\label{lem:containers} For any integers $a, b, g$ and $h$ that satisfy $0 \le b \le h \le a \le g$ and $a \ge n^2$,
\beq{container.ineq} |\cH(a,b,g,h)| \leq 2^{g-\gO((t+t')/\log n)}.\enq
\end{lemma}

For $b, h \in \mathbb N$ and $S \sub L_{n-1}$, let
\[\cG_{n-2}(b,h,S):=\{B \sub L_{n-2}:|[B]|=b, |N(B)|=h, N(B) \sub S\}.\]
Note that, in the above definition, $B$ is not required to be 2-linked.

\begin{lemma}\label{lem:phi'}
    Let $S \sub L_{n-1}$, $1 \le \phi \le n/4$, and $b,h \in \mathbb N$. There is a family $\cV=\cV(b,h,S) \sub 2^{L_{n-1}}$ with
    \[|\cV|\le t'n|S|\exp\left[O\left(\frac{h\log^2 n}{n\phi}+\frac{t' \log^2 n}{n/2-\phi}+\frac{t'\log^2 n}{\phi }\right)\right]\]
    such that each $B \in \cG_{n-2}(b,h,S)$ has a $\phi$-approximation in $\cV$.
\end{lemma}

    The proof of~\Cref{lem:phi'} is very similar to that for \cite[Lemma 8.6]{JMP1} and we defer the details to \ref{app:lem:phi'}.

\begin{lemma}\label{lem:bhs} If $b\ge 1$, then
    \[|\cG_{n-2}(b,h,S)|\le |S|\cdot 2^{h-\gO(t'/\log n)}.\]
\end{lemma}

\begin{proof}[Proof sketch]
    The proof is almost identical to the proof of \Cref{lem:container.graph}; the only difference is that we replace \Cref{lem:phi} by \Cref{lem:phi'} (which results in the replacement of $|Y|$ in \Cref{lem:container.graph} by the significantly smaller $|S|$). Note that we have $h-b\gg \frac{ h\log^3d}{d^2}$ and $h-b \gg \log^2 d$ by the isoperimetric properties guaranteed by \Cref{prop:isoperim_up} and the assumption that $b \ne 0$.
\end{proof}

\begin{proof}[Proof of \Cref{lem:containers}] If $b=0$, then \Cref{lem:containers} reduces to \Cref{lem:container.graph}, so we assume that $b \ge 1$. Set $\phi=n/4$ and $\psi=\sqrt n$.
    Set $\cG_{n-1}(a,g)=\{A \sub L_{n-1}: A \text{ 2-linked, } |[A]|=a, |N(A)|=g\}$ and let $(S,F)$ be a $\psi$-approximation of $A \in \cG_{n-1}(a,g)$. By \Cref{lem:phi} and \Cref{lem:psi}, doing a similar computation that led to \eqref{container.cost'}, the cost for the $\phi$- and the $\psi$- approximations for $\cG_{n-1}(a,g)$ is 
    \beq{eq:G_{n-1}}|L_{n-1}|\exp\left[O\left(\frac{t\log n}{\sqrt n}\right)\right].\enq
    Here we also used stronger isoperimetric conditions than the assumption in \Cref{lem:container.graph'} (i.e., $t \gg g\log^3 d/d^2$) due to \Cref{prop:isoperim_up}.

    Our goal is to show that the cost for specifying $(A,B) \in \cH(a,b,g,h)$ given a $\psi$-approximation $(S,F)$ for $A$ is at most $2^{g-\gO((t+t')/\log n)}$.
  
    We again apply different strategies for specifying $A$ and $B$ depending on the relative sizes of $F$ and $G$. Let $\gd'$ be a small constant, say $\gd'=1/100$. 

    \nin \textit{Case 1.} $g-f \le \gd' t/\log n$

In this case, we first specify $G \setminus F$ as a subset of $N(S)$. Since $S \sub N(F) \sub N(G)$, $|N(S)|\le n^2g$. Therefore, the cost for $G \setminus F$ is at most
\[{n^2g \choose \le \gd' t/\log n}\stackrel{\eqref{binom}}{\le}\exp\left\{\frac{\gd' t}{\log n}\log\left(\frac{en^2g\log n}{\gd' t}\right)\right\}\stackrel{(\dagger)}{\le} \exp\left\{\frac{\gd' t}{\log n}\log\left(\frac{Cn^3\log n}{ \gd'}\right)\right\}\le e^{10\gd' t}<2^{t/2},\]
{where $(\dagger)$ uses \Cref{prop:isoperim_up}(b) and $C$ is a constant.}
Once we specify $G \setminus F$, this determines $G$, thus $[A]$. We then specify $B \in \cG_{n-2}(b,h,[A])$, which costs, by \Cref{lem:bhs}, $a\cdot 2^{h-\gO(t'/\log n)}$. Finally, we specify $A\setminus N(B)$ as a subset of $[A] \setminus N(B)$, which costs $2^{a-h}=2^{g-t-h}.$ In sum, the total cost for specifying $(A,B)$ is at most $a \cdot 2^{t/2+(h-\gO(t'/\log n))+(g-t-h)}\le2^{g-\gO((t+t')/\log n)}$ (the inequality uses the fact that $t=\gO(g/n) \gg \log a \log n$) as desired.

\nin \textit{Case 2.} $g-f>\gd' t/\log n$

In this case, we first specify $B \in \cG_{n-2}(b,h,S)$ (recall that a $\psi$-approximation $(S,F)$ of $A$ is given), which costs $ng \cdot 2^{h-\gO(t'/\log n)}$ (by \Cref{lem:bhs} and the fact that $|S|\le ng$). Then we specify $A\setminus N(B)$ as a subset of $S \setminus N(B)$, which costs $2^{s-h}$. We claim that $2^{s-h} \le 2^{g-\gd' t/(2\log n)-h}$, which would follow once we show that $s \le g-\delta't/(2\log n)$; indeed, by \Cref{prop:SF.facts} (c) and our choice of $\psi$,
$$s\le f+O(\gk/(n\sqrt n)) \stackrel{\eqref{eq:t.kappa}}{\le} f+O(t/\sqrt n) \le g-\gd' t/(2\log n)$$
for large enough $n$. So the total cost for specifying $(A,B)$ in this case is again at most $ng \cdot 2^{g-\gO((t+t')/\log n)}\le2^{g-\gO((t+t')/\log n)}$ (the inequality uses the fact that $t\gg \log (ng) \log n$). \end{proof}

\section{Reduction to the central layers -- proof of \Cref{thm.reduction}}\label{sec:reduction}

We prove \Cref{thm.reduction} by iterating an argument which says that the number of antichains contained in three consecutive layers $L_{[i-2,i]}$ ($i \le n$) is very close to the number of antichains in $L_{[i-1,i]}$. Given $X, Y \sub [3]^n$, let $Y^X=\{v \in Y: v \not< w ~ \forall w \in X\}$.

    \begin{lemma}\label{lem:3to2}
        For $i\leq n$ and any $X\subseteq L_{[i+1,2n]}$,
        \begin{equation}\label{eqn:3to2}
            \alpha(L^X_{[i-2,i]})=\left(1+e^{-\Omega(n^2)}\right)\alpha(L^X_{[i-1,i]}).
        \end{equation}
    \end{lemma}

\Cref{lem:3to2} in fact reduces to the following special case, whose proof is the main contribution of the current section. The reduction uses a nice `embedding trick' from Hamm and Kahn \cite{embed}, which is later adapted in \cite{BalK22, JMP1}.

    \begin{lemma}\label{lem:3to2_mid}
        For any $X\subseteq L_{[n+1,2n]}$,
        \begin{equation}\label{eqn:3to2_mid}
            \alpha(L^X_{[n-2,n]})=\left(1+e^{-\Omega(n^2)}\right)\alpha(L^X_{[n-1,n]}).
        \end{equation}
    \end{lemma}

We first show how \Cref{lem:3to2} follows from \Cref{lem:3to2_mid} and then show how \Cref{thm.reduction} follows from \Cref{lem:3to2}. Finally we prove \Cref{lem:3to2_mid}.

    \begin{proof}[Proof of \Cref{lem:3to2} assuming \Cref{lem:3to2_mid}] Let $i<n$ and $X \sub L_{[i+1,2n]}$ be given. Define $\phi^i:[3]^n \rightarrow [3]^{2n-i}$ to be the operation of appending $n-i$ twos to the given $w \in [3]^n$, i.e., $\phi^i(w)=(w,2,\dots,2)\in [3]^{2n-i}$. For any set $W \sub [t]^n$, let $\phi^i(W)=\{\phi^i(w):w\in W\}$. Note that this operation maps $L_i(n)=L_i(3,n)$ into the middle layer of $[3]^{2n-i}$. 
    
    Let $Y=\{2\cdot\boldsymbol{1}-\mathbf{e}_j: j\geq n+1\} ~(\sub [3]^{2n-i})$
        where $2\cdot\boldsymbol{1} \in [3]^{2n-i}$ denotes the all $2$'s vector and $\mathbf{e}_j \in [t]^{2n-i}$ is the unit vector with its $j$-th coordinate 1. Finally, let $X'=\phi^i(X) \cup Y~(\sub [3]^{2n-i})$.

    \begin{claim}\label{cl:cl}
        For any $j<i$, $\phi^i$ is an order-preserving bijection between $L(n)^X_{[j,i]}$ and $L(2n-i)^{X'}_{[2n-2i+j,2n-i]}$.
    \end{claim}

     Note that it follows from the above claim that 
    $\alpha(L(n)^X_{[i-2,i]})=\alpha(L(2n-i)^{X'}_{[2n-i-2,2n-i]})$ and $\alpha(L(n)^X_{[i-1,i]})=\alpha(L(2n-i)^{X'}_{[2n-i-1,2n-i]})$, so \Cref{lem:3to2_mid} implies \Cref{lem:3to2}.

    \begin{subproof}[Proof of \Cref{cl:cl}] It is immediate from its definition that $\phi^i$ is injective and order-preserving. Also observe that for any $k \in [j,i]$, $\phi^i$ maps $L(n)_k$ into $L(2n-i)_{2n-2i+k}$.

    To see that $\phi^i$ maps $L(n)^X_k$ into $L(2n-i)^{X'}_{2n-2i+k}$, suppose $x \in L(n)_k^X$ and let $y=\phi^i(x)$. We need to show that $y$ is "not below $X'$," meaning that there doesn't exist $x' \in X'$ such that $y \le x'$. For the sake of contradiction, suppose there is $x' \in X'~(=\phi^i(X) \cup Y)$ such that $y \le x'$. First observe that $x'$ cannot be in $Y$, since the last $n-i$ coordinates of $y$ are all twos while those of $x'$ contain a one. So it must be that $x' \in \phi^i(X)$. But then $z:=(\phi^i)^{-1}(x')$ is in $X$ and $z \ge x$, which contradicts that $x$ is not below $X$.

        To see the surjectivity of $\phi^i$, suppose $y \in L(2n-i)^{X'}_{2n-2i+k}$. Since $y$ is not below $X'$, $y$ is not below $Y$. This means that i) the last $n-i$ coordinates of $y$ must all equal 2; and so ii) the sum of the first $n$ coordinates of $y$ is $(2n-2i+k)-2(n-i)=k$. By i) and ii), $y$ is in $\phi^i(L(n)_k)$.
        To finish the proof, we assert that $x:=(\phi^i)^{-1}(y)~(\in L(n)_k)$ is not below $X$. This is true because, if $x$ was below $X$, then since $\phi^i$ is order-preserving, $y=\phi^i(x)$ would have been below $\phi^i(X)$. This contradicts the assumption that $y$ is not below $X'$.
    \end{subproof}
    \end{proof}

\begin{proof}[Proof of \Cref{thm.reduction} assuming \Cref{lem:3to2}] Let $\cA$ denote the collection of antichains of $[3]^n$. Note that, for any $k \in [1, n-1]$ and $\ell \ge k+2$,
\[\begin{split}\alpha(L_{[k-1,\ell]})=\sum_{\substack{X \sub L_{[k+2,\ell]}\\X \in \cA}}\alpha(L^X_{[k-1,k+1]})&\stackrel{\eqref{eqn:3to2}}{=}\left(1+e^{-\gO(n^2)}\right)\sum_{\substack{X \sub L_{[k+2,\ell]}\\X \in \cA}}\alpha(L^X_{[k,k+1]})\\&=\left(1+e^{-\gO(n^2)}\right)\alpha(L_{[k,\ell]}).
\end{split}\]
So by iterating the above argument for $\ell=2n$ and $k=1, \ldots, n-1$, we have
\[\alpha([3]^n)=\left(1+e^{-\gO(n^2)}\right)^{n-1}\alpha(L_{[n-1,2n]}).\]
Notice that, by symmetry between the top and the bottom half of $[t]^n$, $\alpha(L_{[n-1,2n]})=\alpha(L_{[0, n+1]})$. Then again, by iterating the above argument for $\ell=n+1$ and $k=1, \ldots, n-1$, we have
\[\alpha([3]^n)=\left(1+e^{-\gO(n^2)}\right)^{2n-2}\alpha(L_{[n-1,n+1]})=\left(1+e^{-\gO(n^2)}\right)\alpha(L_{[n-1, n+1]}). \qedhere\]
\end{proof}

\subsection{Set-up for the proof of \Cref{lem:3to2_mid}}

We will define two polymer models to express the number of antichains in (\ref{eqn:3to2_mid}) in terms of the corresponding partition functions.

   For $X\subseteq L_{[n+1,2n]}$, let $\cP^X$ be the set of all $2$-linked subsets of $L^X_{n-1}$. Define the compatibility relation on $\cP^X$ where $A_1\sim A_2$ if $A_1\cup A_2$ is \textit{not} $2$-linked. For $A\subseteq L_{n-1}$, we denote by $\textrm{int}(A)$ the set of all $v\in L_{n-2}$ such that $N^+(v)\subseteq A$.
   Define the weight $A\subseteq L_{n-1}$ to be
    \beq{def:polymer.w} w(A)=2^{|\textrm{int}(A)|-|N^+(A)|}\enq
    and let $\Xi_X:=\sum_{\Lam\in\Omega}\prod_{A\in\Lam}w(A)$ denote the corresponding partition function. Recall that $\ell_n$ is the size of the middle layer of $[3]^n$.

    \begin{lemma}\label{lem:Xi-alpha} For any $X\subseteq L_{[n+1,2n]}$,
        $$2^{|L^X_n|}\Xi_X=\al(L^X_{[n-2,n]}).$$
    \end{lemma}
    \begin{proof}
First observe that if $A\subseteq L_{n-1}^X$ has $2$-linked components $A_1,\ldots, A_k$ (i.e.\ $A_1,\ldots, A_k$ is a collection of compatible polymers and $A=A_1\cup\ldots\cup A_k$) then
\begin{align}\label{eq:weightfactor}
w(A)=\prod_{i=1}^kw(A_i)\, .
\end{align}
        Now, every antichain in $L_{[n-2,n]}^X$ can be uniquely constructed as follows: pick any subset $A$ of $L_{n-1}^X$, any $T\subseteq \textrm{int(A)}$ and add $T\cup (A\setminus N^+(T))$ to the antichain. Then, we may add any $S\subseteq L_n^X\setminus N^+(A)$. Hence,
            $$\al(L^X_{[n-2,n]})=\sum_{A\subseteq L^X_{n-1}}2^{|\textrm{int}(A)|}\cdot 2^{|L^X_n|-|N^+(A)|}=2^{|L^X_n|}\Xi_X,$$
        where the last equality follows from~\eqref{eq:weightfactor} and the natural bijection between subsets of $L^X_{n-1}$ and collections of compatible polymers. 
    \end{proof}

        The second polymer model has the same set of polymers and compatibility relation, while the weight is defined as $w'(A):=2^{-|N^+(A)|}$. Let $\Xi'_X$ be its partition function. It is straightforward to verify that, in analogy to~\Cref{lem:Xi-alpha}, we have
    $$2^{|L^X_n|}\Xi'_X=\al(L^X_{[n-1,n]}).$$
    To conclude the proof of Lemma \ref{lem:3to2_mid}, it remains to prove that
    \begin{equation}\label{eqn:ratio}
    \frac{\Xi_X}{\Xi'_X}=1+e^{-\Omega(n^2)}.
    \end{equation}

\begin{observation}
\label{remark}
For all polymers $A\in\cP^X$ we have
   $0<w'(A)\leq w(A)$. Therefore, for any cluster $\Gamma\in \cC(\cP^X,\sim)$ we have that $w(\Gamma)$ and $w'(\Gamma)$ (defined in~\eqref{eq:wclusterdef}) have the same sign and that $|w'(\Gamma)|\leq |w(\Gamma)|.$

\end{observation}

\subsection{Proof of \Cref{lem:3to2_mid}} In this section, we verify \eqref{eqn:ratio}.
Our main task is to establish the convergence of the cluster expansion of $\ln \Xi_X$, for which we now verify the Kotetck\'y-Preiss condition \eqref{KPbound1} for $w(A)=2^{|{\rm int}(A)|-|N^+(A)|}$. Note that the convergence of the cluster expansion of $\ln \Xi'_X$ will follow from that of $\ln \Xi_X$ by \Cref{remark}.

We first recall the following lemma, which follows from the fact (see e.g. \cite[p. 396, Ex. 11]{knuth1968art}) that the infinite $\Delta$-branching rooted tree contains at most $(e\Delta)^{n-1}$ rooted subtrees with $n$ vertices.

    \begin{lemma}\label{lem:linked_count}
        Let $G$ be a graph with maximum degree $\Delta$. The number of $k$-linked subsets of $G$ of size $t$ that contain a fixed vertex $v$ is at most $(e\Delta^k)^{t-1}$.
    \end{lemma}

    \begin{lemma}\label{prop:KP_verif}
        The polymer model $(\cP^X, \sim, w)$ satisfies the Koteck\'y-Preiss condition \eqref{KPbound1} with
        \beq{g.def} f(A)=\frac{|A|}{n^2}\ln 2 \quad \text{and} \quad g(A)=\begin{cases}
         \left(\frac{n-2}{2}|A|-|A|^2\right)\ln2-10|A|\ln n & |A|\leq n/10\\
        \frac{9C}{10}n|A|\ln2 & n/10<|A|\leq n^4\\
        \frac{|A|}{n^2}\ln 2 & |A|>n^4
    \end{cases}.\enq
    with $C:=\min\{c/9,c'/3\}$ where $c,c'$ are the constants in \Cref{prop:isoperim_up}.
    \end{lemma}
\begin{proof}
We first show that for any $u \in L_{n-1}^X$,
\begin{align}\label{eq:case1-3}
     \sum\limits_{A \in \mathcal P^X, A\ni u} w(A)e^{f(A)+g(A)} \leq 1/(2n^5)
\end{align}
by considering the following three cases. 
In each of the case analyses below, the sum ranges over all 2-linked sets $A \sub L_{n-1}$, and this gives an upper bound on the corresponding sum taken over $A \in \mathcal P^X$.

\nin \textit{Case 1.} $|A|\leq n/10$
    
Observe that $\textrm{int}(A)=\emptyset$ in this case. Also, by Proposition \ref{prop:isoperim_up} (a), $|N^+(A)|\geq \frac{n}{2}|A|-|A|^2$.  Therefore, using \Cref{lem:linked_count} for the second inequality (to bound the number of 2-linked $A$'s of size $k$),
    \begin{align*}
        \sum_{\substack{A\ni u, |A|\leq n/10 \\ A\ 2\text{-linked}}} w(A)e^{f(A)+g(A)}&\leq \sum_{\substack{A\ni u, |A|\leq n/10 \\ A\ 2\text{-linked}}} 2^{-\frac{n}{2}|A|+|A|^2}2^{|A|/n^2}2^{\frac{n-2}{2}|A|-|A|^2}n^{-10|A|}\\
        &{\leq} \sum_{k\geq1} \left(en^22^{-1+1/n^2}n^{-10}\right)^k \leq 1/(6n^5).
    \end{align*}

    \nin \textit{Case 2.} $n/10< |A|\leq n^4$
    
    By \Cref{prop:isoperim_up}(c), $|N^+(A)|\geq Cn|A|$ and $|\textrm{int}(A)|\leq |A|/Cn$. Therefore, again using \Cref{lem:linked_count} for the second inequality,
    \begin{align*}
        \sum_{\substack{A\ni u, n/10 < |A|\leq n^4 \\ A\ 2\text{-linked}}}w(A)e^{f(A)+g(A)}&\leq \sum_{\substack{A\ni u, |A|\leq n^4 \\ A\ 2\text{-linked}}} 2^{\frac{|A|}{Cn}-Cn|A|}2^{|A|/n^2}2^{\frac{9C}{10}n|A|}\\
        &\leq\sum_{k\geq1}\left(en^2 2^{\frac{1}{Cn}-\frac{C}{10}n+1/n^2}\right)^k \leq1/(6n^5).
    \end{align*}

    \nin \textit{Case 3.} $|A|>n^4$

    By \Cref{prop:isoperim_up}(b), for any $B \sub L_{n-1}$ or $B \sub L_{n-2}$, we have $|N^+(B)|\ge (1+C/n)|B|$. Noting that $w(A)=2^{-|N^+(A)|+|{\rm int}(A)|}=\sum_{B \sub {\rm int}(A)}2^{-|N^+(A)|}$,
    \begin{align*}
        \sum_{\substack{A\ni u, |A|> n^4 \\ A\ 2\text{-linked}}}w(A)e^{f(A)+g(A)}&=\sum_{\substack{A\ni u, |A|> n^4 \\ A\ 2\text{-linked}}}\sum_{B \sub {\rm int}(A)}2^{-|N^+(A)|}e^{f(A)+g(A)}\\
        &\le \sum_{\substack{a >n^4, ~g \ge (1+C/n)a, \\ 0 \le b \le a, (1+C/n)b \le h \le a}}2^{-g}2^{2a/n^2} |\cH(a,b,g,h)|\\
        &\le \sum_{\substack{a >n^4, ~g \ge (1+C/n)a}}a^22^{2a/n^2}2^{-\Omega((g-a)/\log n)}\\
        &\le \sum_{a >n^4} a^22^{2a/n^2}2^{-\gO(a/(n\log n))} \le 1/(6n^5)
    \end{align*}
where the second inequality uses \Cref{lem:containers} (and the fact that $b, h \le a$); there is plenty room in the last inequality since $a>n^4$. 

Combining the three cases above, we conclude that \eqref{eq:case1-3} holds. 

Now we verify \eqref{KPbound1}. For any polymer $A$, set $N^2(A)\coloneqq N^-(N^+(A))$ and note that $|N^2(A)|\leq n^2|A|$. Using~\eqref{eq:case1-3}, we have
\begin{equation}
\label{KP}
    \sum\limits_{A'\nsim A} w(A')e^{f(A')+g(A')} \leq \sum\limits_{u\in N^2(A)} \sum\limits_{A' \ni u} w(A')e^{f(A')+g(A')} \leq n^2|A| /(2n^5) < f(A)/2\, .
\end{equation}
Therefore, by \Cref{thm.KP} and \Cref{remark}, the cluster expansions of $\ln\Xi$ and $\ln\Xi'$ converge absolutely.
\end{proof}

Finally, to show \eqref{eqn:ratio}, observe that for a cluster $\Gamma \in \cC=\cC(\cP^X,\sim)$ with $\|\Gamma \|<n/2$, we have $w'(\Gamma)=w(\Gamma)$ since the interior of any polymer $A \in \Gamma$ is empty. Therefore, recalling \Cref{remark}, we have
\[|\ln\Xi-\ln\Xi'|\le\sum_{\|\Gamma\|\ge n/2}|w(\Gamma)|.\]
Now we bound the right-hand side of the above. By \eqref{kpbounds}, for any $v\in L_{n-1}$,
\[
    \sum_{\substack{\Gamma \in \cC \\ \Gamma \nsim \{v\}}} |w(\Gamma)|e^{g(\Gamma)}\leq f(\{v\})\le 1/n^2.
\]
Note that if $\lVert\Gamma\rVert \geq n/2$, then $g(\Gamma) = \gO(n^2)$. Therefore, 
\[\begin{split}
\sum_{\lVert\Gamma\rVert \geq n/2}|w(\Gamma)| &\le \sum_{v \in L_{n-1}}\sum_{\substack{\lVert\Gamma\rVert \ge n/2 \\ \Gamma \nsim \{v\}}}|w(\Gamma)|e^{g(\Gamma)}e^{-g(\Gamma)} \\
&\le \sum_{v \in L_{n-1}}e^{-\gO(n^2)}\sum_{\substack{\lVert\Gamma\rVert \ge n/2 \\ \Gamma \nsim \{v\}}}|w(\Gamma)|e^{g(\Gamma)}\le 3^n\cdot e^{-\gO(n^2)}\cdot 1/n^2=e^{-\gO(n^2)}.
\end{split}\]
This yields \eqref{eqn:ratio}.

\section{Antichains in three central layers -- proof of Theorems~\ref{thm:asymp_1} and \ref{thm.central}} \label{sec:central}

In this section, we study the number and structure of antichains in the central layers $L_{[n-1, n+1]}$ and deduce Theorems~\ref{thm:asymp_1} and \ref{thm.central}.  We will again rely on an appropriate polymer model and the cluster expansion.

    Let $G_\ell=L_{n-1}\cup L_n$ and $G_u=L_n\cup L_{n+1}$ ($\ell,u$ for lower and upper). The set of polymers $\cP_C$ consists of all sets $A\subseteq L_{n-1}$ or $A\subseteq L_{n+1}$ that are $2$-linked. The weight of a polymer is $w_C(A):=2^{-|\partial(A)|}$, where we write $\partial(A)$ for the set of vertices in $L_n$ that are comparable to at least one element in $A$ (i.e., the neighborhood of $A$ intersected with $L_n$). Define the compatibility relation on $\cP_C$ where $A_1\sim A_2$ if $\partial(A_1)\cap \partial(A_2)=\emptyset$. Let $\Xi_C$ be the partition function of the corresponding polymer model. We first prove this model accurately captures the number of antichains in the three central layers.

    \begin{lemma}\label{lem:Xi_C}        $$2^{\ell_n}\Xi_C=\alpha(L_{[n-1,n+1]}).$$
    \end{lemma}
    \begin{proof}
        Every antichain in $L_{[n-1,n+1]}$ can be uniquely constructed as follows: pick any pair of subsets 
        \beq{eq:Aud} \text{$A_u\subseteq L_{n+1}, A_\ell\subseteq L_{n-1}$ with disjoint neighborhoods in $L_n$,}\enq which guarantees that $A=A_u\cup A_\ell$ is an antichain. Then, we may add any $S\subseteq L_n\setminus \partial(A)$. Hence
        $$\alpha(L_{[n-1,n+1]})=\sum_{A_u,A_\ell}2^{|L_n|-|\partial(A_u\cup A_\ell)|}=2^{|L_n|}\Xi_C,$$
        where the sum is over all pairs $A_u,A_\ell$ that satisfy \eqref{eq:Aud}. The last equality follows from the natural bijection between collections of compatible polymers and sets $A_u\cup A_\ell$ where $A_u,A_\ell$ satisfy \eqref{eq:Aud}.
    \end{proof}

    \begin{lemma}\label{lem:central.converge}
        The polymer model $(\cP_C, \sim, w_C)$ satisfies the Kotetck\'y-Preiss condition \eqref{KPbound1} with $f,g$ as in \eqref{g.def}. Thus the cluster expansion of $\ln\Xi_C$ converges absolutely, and
        \beq{eq.wc} \sum_{\Gam\in \cC}|w_C(\Gam)|e^{g(\Gam)}\leq 3^n/n^2.\enq
    \end{lemma}
    \begin{proof}
        We want to show that for any $A\in \cP_C$ we have
        $$\sum_{B\not\sim A}w_C(B)e^{f(B)+g(B)}\leq f(A).$$
        Note that any polymer that is not compatible with $A$ must contain a vertex which has Hamming distance exactly 2 from $A$.
        Hence, with $N^2(A)$ denoting the set of vertices that have Hamming distance 2 from $A$ (i.e., the second neighborhood of $A$ in the Hasse diagram of $L_{[n-1, n+1]}$),
        $$\sum_{B\not\sim A}w_C(B)e^{f(B)+g(B)}\leq \sum_{v\in N^2(A)}\sum_{\substack{B \in \cP_C \\ B\ni v}}w_C(B)e^{f(B)+g(B)}.$$
        If $v \in N^2(A)$ in the outer sum is in $L_{n-1}$, then $B\subseteq L_{n-1}$, and so in this case
        $$\sum_{\substack{B\in \cP_C \\ B\ni v}}w_C(B)e^{f(B)+g(B)}\leq \sum_{\substack{B\in \cP^{\emptyset} \\ B\ni v}}w(B)e^{f(B)+g(B)}\stackrel{\eqref{eq:case1-3}}{\leq} \frac{1}{2n^5}.$$
        The first inequality holds because $\{B \in \mathcal P_C:B \ni v\}=\{B \in \mathcal P^\emptyset: B \ni v\}$, and for $B$ in this set (recalling \eqref{def:polymer.w}) $w_C(B) \le w(B)$.
        We have the same conclusion for $v\in L_{n+1}$ using the symmetry via the map $u\mapsto 2\cdot\boldsymbol{1}-u$. Now, noting that $|N^2(A)| \le 2n^2|A|$ (counting second neighbors on the same layer and two layers above),
        $$\sum_{\substack{B\nsim A}}w_C(B)e^{f(B)+g(B)}\le \sum_{v\in  N^2(A)}\sum_{\substack{B \in \cP_C \\ B \ni v}} w_C(B)e^{f(B)+g(B)} \le 2n^2|A|/(2n^5)\le f(A).$$
        Since the Koteck\'y-Preiss condition is satisfied, the cluster expansion of $\ln \Xi_C$ converges absolutely.

        To obtain \eqref{eq.wc}, we apply \eqref{kpbounds} with $\gamma=\{v\}$ for any $v$, from which we have
        $$\sum_{\substack{\Gam\in\cC \\ \Gam\not\sim\{v\}}}|w_C(\Gam)|e^{g(\Gam)}\leq f(\{v\})\le 1/n^2,$$
        and sum over all $v\in L_{n-1}\cup L_{n+1}$.
    \end{proof}

{Having established convergence of the cluster expansion of $\ln \Xi_C$, it will be useful to record the following facts about the contribution to the expansion from small clusters. 
\begin{proposition}\label{prop:cluster_1} We have
\begin{enumerate}[(i)]
   \item \label{eq:cluster_1} \[
        \sum_{\|\Gamma\|=1}w_C(\Gamma)=2\sum_{0 \le k < n/2}\binom{n}{2k+1}\binom{2k+1}{k}2^{-(n-k)}~(=T_1)\, ,
   \]
   \item \label{eq:cluster_2}
\[
    \sum_{\substack{\Gam\in\cC \\ \|\Gam\|=2}}w_C(\Gam)=\sum_{0 \le k<n/2}\binom{n}{2k+1}\binom{2k+1}{k}\left(n^2-3(k+1)n+\frac{k(9k+17)}{4}\right)2^{-2(n-k)}(=T_2)\, ,\]
\item \label{eq:cluster_3} \[
  \sum_{\|\Gam\|=3}|w_C(\Gam)|=e^{-\Omega(n)}\, ,
\]
\item \label{eq:cluster_4}\[\sum_{\|\Gam\|=2}w_C(\Gam)=o\left(\sum_{\|\Gam\|=1}w_C(\Gam)\right)\, .\]
\end{enumerate}
\end{proposition}
The proof of~\Cref{prop:cluster_1} requires some detailed calculations which we defer to~\Cref{sec:computations}.}

We now define a new weight function $w'_C$ on $\cP_C$ as  follows: $w'_C(A)=w_C(A)$ if $|A| \le 2$, and $w'_C(A)=0$ otherwise. Let $\Xi_C'$ be the corresponding partition function. In analogy with~\Cref{lem:Xi_C}, it is straightforward to verify that
\[
2^{|L_n|}\Xi_C'=|\cA_0|
\]
where $\cA_0$ is defined as in~\Cref{thm.central}.
The first conclusion of \Cref{thm.central} now follows from~\Cref{lem:Xi_C} and the lemma below.

\begin{lemma}
    \[\Xi_C=\left(1+e^{-\gO(n)}\right)\Xi'_C.\]
\end{lemma}

\begin{proof}
    Note that if a cluster $\Gamma \in \cC(\cP_C, \sim)$ only contains polymers of size at most 2, then $w'_C(\Gamma)=w_C(\Gamma)$; otherwise, $w_C'(\Gamma)=0$. Therefore,
    \beq{eq:xi.vs.xi'}\ln\Xi_C-\ln\Xi'_C\le\sum_{\|\Gamma\|=3}|w_C(\Gamma)|+ \sum_{\|\Gamma\|\ge 4}|w_C(\Gamma)|.\enq

    The first term of the right-hand side of \eqref{eq:xi.vs.xi'} is $e^{-\Omega(n)}$ by  \Cref{prop:cluster_1}~\eqref{eq:cluster_3}. For the second term, we show that, for any fixed $\ell \ge 1$,
    \beq{eq.gamma.4} \sum_{\Gamma \in \cC: \|\Gamma \| \ge \ell} |w_C(\Gamma)| \le \frac{3^n n^{10\ell}}{n^2 2^{\frac{n-2}{2}\ell-\ell^2}},\enq
    which will show that $\sum_{\|\Gamma \|\ge 4}|w_C(\Gamma)|=e^{-\gO(n)}.$ To this end, we analyze the cluster expansion of $\ln \Xi_C$, which is convergent by \Cref{lem:central.converge}. Note that with the choice of function $g$ as in \eqref{g.def}, since $g(x)$ is increasing for $x <(1-o(1))n/4$, we have that $\|\Gamma \| \ge \ell$ implies $g(\Gamma) \ge g(\ell)$. Combining this with \eqref{eq.wc},
        \[e^{g(\ell)}\sum_{\Gamma \in \cC:\|\Gamma\|\ge \ell}|w_C(\Gamma)|\le \sum_{\Gamma \in \cC:\|\Gamma\|\ge \ell}|w_C(\Gamma)|e^{g(\Gamma)}\le 3^n/n^2,\]
        and \eqref{eq.gamma.4} follows by plugging in $ g(\ell)=\left(\frac{n-2}{2}\ell-\ell^2\right)\ln 2-10\ell \ln n$ on the left-hand side.
\end{proof}

To prove the second part of \Cref{thm.central}, in view of \Cref{lem:Xi_C}, it remains to show the following.

\begin{lemma}\label{lem:XiCasym}
    $$\Xi_C=(1+e^{-\Omega(n)})\exp{(T_1+T_2)}.$$
\end{lemma}
\begin{proof}
    We know by \Cref{lem:central.converge} that the cluster expansion of $\ln \Xi_C$ converges. Hence, it suffices to show that
    $$\sum_{\|\Gamma \|=1}w_C(\Gamma)=T_1, \sum_{\|\Gamma \|=2}w_C(\Gamma)=T_2, \sum_{\|\Gamma \|=3}|w_C(\Gamma)|=e^{-\Omega(n)}\text{ and } \sum_{\|\Gamma \|\geq4}w_C(\Gamma)=e^{-\Omega(n)}.$$
    The last one follows from \eqref{eq.gamma.4}, while the first three follow from \Cref{prop:cluster_1}.
\end{proof}

Finally we note that \Cref{thm:asymp_1} follows by combining~\Cref{lem:XiCasym}, \Cref{lem:Xi_C} and Theorem~\ref{thm.reduction}.

\section{Proof of \Cref{thm:CLT}}\label{sec:CLT}
Recall that, given an antichain $I \sub [3]^n$, we call the vertices of $I\backslash L_n$ the defect vertices of $I$ and we let $X_n$ denote the number of defect vertices in a uniformly randomly chosen antichain in $[3]^n$. Our goal is to prove a CLT for $X_n$. To this end, it will be more convenient to work with the random variable $Y_n$ which counts the number of defect vertices in a uniformly randomly chosen antichain in the middle three layers $L_{[n-1,n+1]}$. We first show that $Y_n$ satisfies a CLT from which the CLT for $X_n$ will follow easily. We therefore turn our attention to proving the following proposition.

\begin{proposition}\label{prop:CLT}
     The random variable $\tilde Y_n := (Y_n-\E(Y_n))/\sqrt{\textup{Var} (Y_n)}$
    converges in distribution to the standard normal $N(0,1)$ as $n \rightarrow \infty.$ Furthermore, 
    $$\E(Y_n)=(1+o(1)){T_1} \hspace{1cm} \text{and} \hspace{1cm} \textup{Var} (Y_n)=(1+o(1)){T_1}.$$
\end{proposition}

 By \Cref{cumulantfact}, in order to derive the first conclusion of \Cref{prop:CLT}, it suffices to show that $\gk_1(\tilde Y_n)\rightarrow 0$, $\gk_2(\tilde Y_n)\rightarrow 1$, and $\gk_k(\tilde Y_n)\rightarrow 0$ for all $k \ge 3$ as $n \rightarrow \infty$. Note that, by the definition of $\tilde Y_n$  and \eqref{cumlin}, we have that, for all $n$, $\gk_1(\tilde Y_n)= 0$, $\gk_2(\tilde Y_n)=1$, and $\gk_k(\tilde Y_n)=(\text{Var}(Y_n))^{-k/2}\gk_k(Y_n)$ for $k\ge 3$.

We define a family of auxiliary polymer models that will be useful for the rest of the section. Recall the definition of the polymer model $(\cP_C,\sim,w_C)$ from the previous section. For $t\in \mathbb{R}$, we define the `tilted' weight $\tilde w_t(S)=w_C(S)e^{t|S|}$. We let $\tilde\Xi_t$ denote the partition function of $(\cP_C,\sim,\tilde w_t)$. Note that \textit{as formal power series}
\begin{equation}\label{eqn:aux_ce}
\ln\tilde\Xi_t=\sum_{\Gam\in \cC}\tilde w_t(\Gam)=\sum_{\Gam\in\cC}w_C(\Gam)e^{t\|\Gam\|}
\end{equation}
and
\begin{align}\label{eqn:aux_deriv}
    \frac{\partial^k\ln \tilde \Xi_t}{\partial t^k}=\sum_{\Gam\in\cC}w_C(\Gam)\|\Gam\|^ke^{t\|\Gam\|}.
\end{align}

We record the following lemma which establishes~\eqref{eqn:aux_ce} around $t=0$. Its proof is straightforward by repeating the proofs of \Cref{prop:KP_verif} and \Cref{lem:central.converge}, and the derivation of \eqref{eq.gamma.4}, so we omit it.

\begin{lemma}\label{prop:polym_aux}
For positive integer $n$, let $t_0=t_0(n):=1/n^3$. There exists an $n_0$ for which the following holds for any $n \ge n_0$: the polymer model $(\cP_C,\sim,\tilde w_{t_0})$ satisfies the Koteck\'y-Preiss condition with the same functions $f$ and $g$ as in \eqref{g.def}. Since $\tilde w_t$ is increasing in $t$, this implies the same conclusion holds for all $t \le t_0$. Hence, the cluster expansion of $\ln\tilde\Xi_t$ converges absolutely in $(-\infty, t_0]$.
\end{lemma}

For future reference, we record that \Cref{prop:polym_aux} implies bound \eqref{eq.wc}, and thus \eqref{eq.gamma.4}, holds for the weights $\tilde w_t$ when $n\geq n_0$ and $t\leq 1/n^3$.

Next we record the following lemma which establishes~\eqref{eqn:aux_deriv} around $t=0$.

\begin{lemma}\label{lem:unif_conv}
    Let $n_0$ be as in \Cref{prop:polym_aux}. For all $n \ge n_0$, the following holds. Let $k\geq1$ be a fixed integer. Then
    $$\sum_{\Gam\in\cC}\tilde{w}_t(\Gam)\|\Gam\|^k$$
    converges uniformly for $t\in [-1/n^3,1/n^3]$. Furthermore,
    \begin{equation}\label{eq:swap_deriv}
    \frac{\partial^k}{\partial t^k}\sum_{\Gamma \in \cC}\tilde w_t(\Gamma)=\sum_{\Gam\in\cC}w_C(\Gam)\|\Gam\|^ke^{t\|\Gam\|}
    \end{equation}
    and 
    \begin{equation}\label{eq:cumulant_tail}
        \sum_{\|\Gam\|\geq4}|\tilde w_t(\Gam)|\|\Gam\|^k=e^{-\Omega(n)}\, .
    \end{equation}
    for $t\in [-1/n^3,1/n^3]$.
\end{lemma}
We defer the proof of~\Cref{lem:unif_conv} to later in the section. For now, we record the following consequence.

\begin{proposition}\label{prop:cumulant} For any fixed integer $k \ge 1$,
    \begin{align}\label{eq:propcumulant}
    \gk_k(Y_n)=\sum_{\Gamma \in \cC} w_C(\Gamma)\|\Gamma\|^k=(1+o(1))T_1\, ,
    \end{align}
    as $n\to\infty$.
\end{proposition}

\begin{proof} We begin by justifying the first equality. For a polymer configuration $\Lam$, set $\nu(\Lam):=\prod_{S \in \Lam}w_C(S)/\Xi_C$.
    Then, since each polymer configuration $\Lam$ occurs with probability $\nu(\Lam)$ as the defects of a uniformly randomly chosen antichain from $L_{[n-1,n+1]}$,
    \[\begin{split}\E e^{tX_n}=\sum_{\Lam \in \gO} \nu(\Lam)e^{t\sum_{S \in \Lam}|S|}=\frac{1}{\Xi_C}\sum_{\Lam \in \gO}\prod_{S \in \Lam} \left(w_C(S)e^{t |S|}\right)=\frac{\tilde \Xi_t}{\Xi_C}.\end{split}\]  
Thus, 
\[\gk_k(Y_n)=\left.\frac{\partial^k}{\partial t^k}\ln \frac{\tilde \Xi_t}{\Xi_C}\right|_{t=0}=\left.\frac{\partial^k\ln \tilde \Xi_t}{\partial t^k}\right|_{t=0}=\left.\frac{\partial^k}{\partial t^k}\sum_{\Gamma \in \cC}\tilde w_t(\Gamma)\right|_{t=0}=\sum_{\Gamma \in \cC}w_C(\Gamma)\|\Gamma\|^k, \]
where for the final two equalities we used~\Cref{prop:polym_aux} and ~\Cref{lem:unif_conv}.

We now turn to the second equality in~\eqref{eq:propcumulant}.
    If we consider only the clusters of size 1, we get a contribution to the sum in ~\eqref{eq:propcumulant} of exactly $T_1$ (which is exponential in $n$) by \Cref{prop:cluster_1} \eqref{eq:cluster_1}. Hence, it remains to show that the contribution of all other terms are dwarfed by this term. First, for fixed $k$ and for clusters $\Gamma$ of size $2$ or $3$, $\|\Gam\|^k=O(1)$, so \Cref{prop:cluster_1}
     yields
    \[\sum_{\Gamma \in \cC, \|\Gamma\|=2,3}w_C(\Gamma)\|\Gamma\|^k=o(T_1). \]
    For the remaining terms, we note that the second conclusion of \Cref{lem:unif_conv} for $t=0$ gives
    \[\sum_{\Gamma \in \cC, \|\Gamma\|\geq4}|w_C(\Gamma)|\|\Gamma\|^k=o(1)=o(T_1). \qedhere\]
\end{proof}

We are now ready to prove \Cref{prop:CLT} (assuming~\Cref{lem:unif_conv}).

\begin{proof}[Proof of \Cref{prop:CLT}]
By \Cref{prop:cumulant}  we obtain  \[\E(Y_n)=\kappa_1(Y_n)=(1+o(1))T_1, \quad \textup{Var}(Y_n)=\kappa_2(Y_n)=(1+o(1))T_1\]
and
\[\gk_k(Y_n)=(1+o(1))\textup{Var}(Y_n) \quad \text{for all $k\ge 3$}.\]
It follows that for any $k \ge 3$ we have $\gk_k(\tilde Y_n)=\kappa_k(Y_n)(\textup{Var}(Y_n))^{-k/2}=O((\textup{Var}(Y_n))^{1-k/2})$, which tends to 0 since $\textup{Var}(Y_n) \rightarrow \infty$ as $n \rightarrow \infty$.
The result now follows from \Cref{cumulantfact}.
\end{proof}

It remains to prove~\Cref{lem:unif_conv}. For this, we use the following elementary fact (see, e.g., \cite[Chapter 24, Corollary of Theorem 3]{spivak2008calculus}).
\begin{lemma}\label{prop:swap_deriv}
    Let $\{f_n\}_{n\geq1}$ be a sequence of real-valued functions defined on a closed interval $[a,b]$.
    If $\sum_{n\geq1} f_n$ converges pointwise on $f$ on $[a,b]$, each $f_n$ has an integrable derivative $f_n'$ and $\sum_{n\geq1}f_n'$ converges uniformly on $[a,b]$, then $f'(t)=\sum_{n\geq1}f'_n(t)$, i.e.
    $$\frac{\partial}{\partial t}\sum_{n\geq 1}f_n(t)=\sum_{n\geq1}\frac{\partial}{\partial t}f_n(t),$$
    for all $t\in[a,b]$.
\end{lemma}

\begin{proof}[Proof of~\Cref{lem:unif_conv}]
 For the first conclusion, as each term $\tilde{w}_t(\Gam)\|\Gam\|^k$ is increasing in $t$ in absolute value, by the Weierstrass $M$-test, it suffices to show the series absolutely converges for $t=t_0:=1/n^3$, i.e.
$$\sum_{\Gam\in\cC}|\tilde w_{t_0}(\Gam)|\|\Gam\|^k<\infty.$$
To show this, it will suffice to establish \eqref{eq:cumulant_tail} for $t=t_0$, since the the contribution to the sum from clusters $\Gamma\in \cC$ with $\|\Gamma\| \le 3$ is finite. Recall from the remark below \Cref{prop:polym_aux} that \eqref{eq.wc} and \eqref{eq.gamma.4} hold for the weight {function} $\tilde w_{t_0}$, that is,
\begin{equation}\label{eq.wt}
    \sum_{\Gam\in \cC}|\tilde w_{t_0}(\Gam)|e^{g(\Gam)}\leq 3^n/n^2
\end{equation}
and
\begin{equation}\label{eq.gammat.4}
    \sum_{\|\Gam\|\geq\ell}|\tilde w_{t_0}(\Gam)|\leq \frac{3^n n^{10\ell}}{n^2 2^{\frac{n-2}{2}\ell-\ell^2}}.
\end{equation}
For $4\leq\|\Gam\|\leq n^4$, 
$$\sum_{4\leq\|\Gam\|\leq n^4}|\tilde w_{t_0}(\Gam)|\|\Gam\|^k\leq n^{4k}\sum_{ 4\leq\|\Gam\|\leq n^4}|\tilde w_{t_0}(\Gam)|\leq n^{4k}\sum_{\|\Gam\|\geq 4}|\tilde w_{t_0}(\Gam)|\overset{\eqref{eq.gammat.4}}{=}e^{-\Omega(n)}.
$$
If $\|\Gam\|> n^4$ then $g(\Gam)= \frac{\|\Gam\|}{n^2}\ln2\geq 2\ln\|\Gam\|^k$, so
$$\sum_{\|\Gam\|> n^4}|\tilde w_{t_0}(\Gam)|\|\Gam\|^k\leq \sum_{\|\Gam\|> n^4}|\tilde w_{t_0}(\Gam)|e^{g(\Gam)/2}.$$
On the other hand, $g(\Gam)= \frac{\|\Gam\|}{n^2}\ln2\geq n^2\ln 2$, so
$$\sum_{\|\Gam\|> n^4}|\tilde w_{t_0}(\Gam)|e^{g(\Gam)/2}\le 2^{-n^2/2}\sum_{\|\Gam\|> n^4}|\tilde w_{t_0}(\Gam)|e^{g(\Gam)} \stackrel{\eqref{eq.wt}}{\leq} 2^{-n^2/2}\frac{3^n}{n^2}=e^{-\Omega(n)}.$$
Hence,
\[\sum_{\|\Gam\|\geq4}|\tilde w_{t_{0}}(\Gam)|\|\Gam\|^k=e^{-\Omega(n)} \]
as desired.

Finally, \eqref{eq:swap_deriv} now follows from \Cref{prop:swap_deriv}.
\end{proof}

Finally we deduce~\Cref{thm:CLT}. Let $I$ be an antichain in $[3]^n$ sampled uniformly at random and let $\mathcal{E}_n$ denote the event that $I \subseteq L_{[n-1,n+1]}$.
Note that $\mathbb{P}(Y_n=i)=\mathbb{P}(X_n=i \mid \mathcal{E}_n)$ for all $i$. Moreover, by~\Cref{thm.reduction}
\[
\mathbb{P}(\mathcal{E}_n)=1-e^{-\gO(n^2)}\, .
\]
Thus, for all $i$,
\begin{align}\label{eq:TVdist}
\mathbb{P}(X_n=i)
=\mathbb{P}(X_n=i \mid \mathcal{E}_n)\mathbb{P}(\mathcal{E}_n) + \mathbb{P}(X_n=i \mid \mathcal{E}^c_n)\mathbb{P}(\mathcal{E}^c_n)
= \mathbb{P}(Y_n=i) + O(e^{-\gO(n^2)})\, .
\end{align}
For brevity, let $\nu_n=\E(Y_n), \tau^2_n=\text{Var}(Y_n)$ and recall that we let $\mu_n=\E(X_n), \sigma^2_n=\text{Var}(X_n)$. Since $X_n,Y_n\leq 3^n$, we have
\[
|\mu_n-\nu_n|\leq \sum_{i=0}^{3^n}i\cdot |\mathbb{P}(X_n=i)-\mathbb{P}(Y_n=i)|\overset{\eqref{eq:TVdist}}{=} o(1)\, .
\]
We similarly have $|\E(Y^2_n)-\E(X^2_n)|=o(1)$ and so $|\tau^2_n-\sigma^2_n|=o(1)$. It follows that
\begin{align}\label{eq:TV2}
\sigma_n/\tau_n-1=o(1) \quad\text{and}\quad (\mu_n-\nu_n)/\tau_n=o(1)
\end{align}
since $\tau_n\to \infty$ by~\Cref{prop:CLT}.

We conclude that for $a\in \mathbb{R}$,
\[
\mathbb{P}\left(\frac{X_n-\mu_n}{\sigma_n}\leq a \right) \overset{\eqref{eq:TVdist}}{=} \mathbb{P}\left(\frac{Y_n-\mu_n}{\sigma_n}\leq a \right)+o(1)\overset{\eqref{eq:TV2}}{=} \mathbb{P}\left(\frac{Y_n-\nu_n}{\tau_n}\leq a +o(1) \right) +o(1)\, .
\]
\Cref{thm:CLT} now follows from \Cref{prop:CLT}.

\section{Contribution from small clusters to the expansion of $\ln \Xi_C$}\label{sec:computations}
In this section we prove~\Cref{thm:asymp_2} and \Cref{prop:cluster_1}. We begin with the proof of ~\Cref{thm:asymp_2} (assuming \Cref{prop:cluster_1}) and then conclude with the proof of \Cref{prop:cluster_1}.

\begin{proof}[Proof of~\Cref{thm:asymp_2}]
By~\Cref{thm:asymp_1} and \Cref{prop:cluster_1} it suffices to show that 
   \begin{equation}\label{eq:cluster_1_asympt}
    T_1=(1+o(1))\sqrt{\frac{1+2\sqrt{2}}{2\sqrt{2}\pi}}n^{-1/2}\left(\frac{1+2\sqrt{2}}{2}\right)^n.
    \end{equation}
 To this end, note that 
    $$T_1=2^{1-n}\sum_{0 \le k<n/2}\binom{n}{2k+1}\binom{2k+1}{k}2^{k}\, ,$$
    and so
    $$T_1=n2^{1-n} \cdot {}_2F_1\left(\frac{1}{2}-\frac{n}{2},1-\frac{n}{2},2;8\right)$$
    where ${}_2F_1\left(a,b,c;z\right)$ is the \textit{Gauss hypergeometric function}. Using \cite[equation (39)]{cvitkovic2017asymptotic}\footnote{\cite{cvitkovic2017asymptotic}  generally assumes $\varepsilon\neq 1$ where $\varepsilon$ is the parameter in \cite[equation (39)]{cvitkovic2017asymptotic}, but the proof (in Sections 2.3.1 and 2.3.2 of \cite{cvitkovic2017asymptotic}) of the formula in \eqref{eq:hyper_asymp}  carries out in the case $\varepsilon=1$.}, we obtain the asymptotic formula
    \begin{equation}\label{eq:hyper_asymp}
    _2F_1\left(a-\lam,b-\lam,c;z\right)=(1+o(1))\Gamma(c)\frac{\sqrt[4]{z}}{2\sqrt{\pi}}\lam^{\frac{1}{2}-c}\cdot\frac{\left(1+\frac{1}{\sqrt{z}}\right)^{c-a+\lam}}{\left(\frac{1}{\sqrt{z}}\right)^{\lam-a}(1+\sqrt{z})^{b-\lam}}\, ,\qquad \lam\rightarrow\infty.
    \end{equation}
    Plugging in $a=1/2$, $b=1$, $c=2$, $z=8$, and $\lam=n/2$,
    $$_2F_1\left(\frac{1}{2}-\frac{n}{2},1-\frac{n}{2},2;8\right)=(1+o(1))\sqrt{\frac{1+2\sqrt{2}}{8\sqrt{2}\pi}}\frac{(1+2\sqrt{2})^n}{n^{3/2}}.$$
   \Cref{eq:cluster_1_asympt} now follows, completing the proof. 
\end{proof}

We now turn our attention to proving \Cref{prop:cluster_1}.
Henceforth, $v$ will be a vertex in $L_{n-1} \cup L_{n+1}$, and $d_v$ will denote the number of neighbors (in the Hasse diagram of $[3]^n$) of $v$ in $L_n$ (that is, the ``up-degree" of $v\in L_{n-1}$ and the ``down-degree" of $v\in L_{n+1}$). The layer in which $v$ lies, will always be clear from the context. For a vertex $v\in L_{n-1}$, we set $N^2_{n+1}(v):=N^+(N^+(v))$, $N^2_{n-1}(v):=N^-(N^+(v))\setminus\{v\}$ and $\hat{N}^2_{n-1}(v):=N^-(N^+(v))$.

\begin{proof}[Proof of \Cref{prop:cluster_1}~\eqref{eq:cluster_1}]   
    We have
    \begin{equation}\label{eq:cluster_1_deg}
    \sum_{\|\Gamma\|=1}w_C(\Gamma)=2\sum_{v\in L_{n-1}} 2^{-d_v},    
    \end{equation}
    where the factor of $2$ accounts for the symmetric case where the vertex is in $L_{n+1}$.
    
 Note that the coordinates of $v \in L_{n-1}$ consist of $k+1$ zeros, $n-(2k+1)$ ones, and $k$ twos for some integer $k \in [0, (n-1)/2]$. Also, if $v$ has $k$ twos, then $d_v=n-k$. Hence, we may rewrite the sum on the right hand side of~\eqref{eq:cluster_1_deg} as
   \[\sum_{0 \le k < n/2}\binom{n}{k}\binom{n-k}{k+1}2^{-(n-k)}=\sum_{0 \le k < n/2}\binom{n}{2k+1}\binom{2k+1}{k}2^{-(n-k)},\]
   as desired.
\end{proof}

\begin{proof}[Proof of \Cref{prop:cluster_1}~\eqref{eq:cluster_2}] 

Clusters $\Gamma$ of size $2$ either consist of two incompatible polymers of size $1$ or a single polymer of size $2$. Recall from \eqref{eq:wclusterdef} the definition of the weight of a cluster.

\nin \textit{Case 1:} $\Gam=(\{u\},\{v\})$ with $u,v\in L_{n-1}$ and $N^+(u)\cap N^+(v)\not=\emptyset$ (and possibly $u=v$). 

The Ursell function $\phi(G_\Gamma)$ is $-1/2$ and $\prod\limits_{{\gamma} \in \Gamma} w_C(\gamma)=2^{-d_u-d_v}$. These clusters contribute
    $$-\frac{1}{2}\sum_{u\in L_{n-1}}\sum_{v\in \hat{N}_{n-1}^2(u)}2^{-d_u-d_v}$$
    to the sum in~\eqref{eq:cluster_2}
    (there is no double counting because the pair $(\{u\},\{v\})$ is ordered).

\nin \textit{Case 2:} $\Gam=(\{u\},\{v\})$ with $u\in L_{n-1},\ v\in L_{n+1}$ and $N^+(u)\cap N^-(v)\not=\emptyset$. 

Again $\phi(G_\Gamma)=-1/2$ and $\prod\limits_{{\gamma} \in \Gamma} w_C(\gamma)=2^{-d_u-d_v}$. These clusters contribute
    $$-\frac{1}{2}\sum_{u\in L_{n-1}}\sum_{v\in N_{n+1}^2(u)}2^{-d_u-d_v}.$$

\nin \textit{Case 3:} $\Gam=(\{u,v\})$ with $\{u,v\}$ a $2$-linked subset of $L_{n-1}$.

This again implies that $N^+(u)\cap N^+(v)\not=\emptyset$, but excludes the case $u=v$. The Ursell function is $1$ and $w_C(\{u,v\})=2^{-|N^+(\{u,v\})|}=2^{-d_u-d_v+1}$ (since in the graphs induced by any two consecutive layers of $[3]^n$, any two vertices have at most one common neighbor). This contributes
    $$\frac{1}{2}\sum_{u\in L_{n-1}}\sum_{v\in N_{n-1}^2(u)}2^{-d_u-d_v+1}=\sum_{u\in L_{n-1}}\sum_{v\in N^2_{n-1}(u)}2^{-d_u-d_v}$$
    (the $1/2$ in front this time is due to double-counting $\{u,v\}$ and $\{v,u\}$).

    In all cases, we have an identical contribution from those clusters where the role of $L_{n-1}, L_{n+1}$ are swapped. Putting everything together, we therefore have
    \begin{align*}
        \sum_{\substack{\Gam\in\cC \\ \|\Gam\|=2}}w_C(\Gam)&=2\sum_{u\in L_{n-1}}\sum_{v\in N^2_{n-1}(u)}2^{-d_u-d_v}-\sum_{u\in L_{n-1}}\sum_{v\in \hat{N}_{n-1}^2(u)}2^{-d_u-d_v}-\sum_{u\in L_{n-1}}\sum_{v\in N_{n+1}^2(u)}2^{-d_u-d_v}\\
        &=\underbrace{\sum_{u\in L_{n-1}}\sum_{v\in N^2_{n-1}(u)}2^{-d_u-d_v}}_{A_1}-\underbrace{\sum_{u\in L_{n-1}}2^{-2d_u}}_{A_2}-\underbrace{\sum_{u\in L_{n-1}}\sum_{v\in N_{n+1}^2(u)}2^{-d_u-d_v}}_{A_3}.
    \end{align*}
Rewrite
$$
A_1=\sum_{0\le k<n/2}\left[\binom{n}{2k+1}\binom{2k+1}{k}2^{-(n-k)}\sum_{v\in N^2_{n-1}(u)}2^{-d_v}\right],
$$
where we grouped vertices $u$ according to the number of twos they have; recall that when $u$ has $k$ twos, it has $n-2k-1$ ones and $k+1$ zeros. For brevity, we will say that the \textit{type} of $u$ is $(k+1,n-2k-1,k)$, and if a vertex is of type $(a,b,c)$ then its down-degree is $n-a$ and its up-degree is $n-c$. For any such $u$, we have the following cases for how $v\in N^2_{n-1}(u)$ is obtained from $u$:
\begin{itemize}
    \item We added $1$ to a zero coordinate and subtracted $1$ from a two coordinate of $u$. There are $k(k+1)$ choices for the coordinates; then $v$ is of type $(k,n-2k+1,k-1)$ and $d_v=n-k+1$.
    \item We added $1$ to a zero coordinate and subtracted $1$ from a one coordinate of $u$. There are $(k+1)(n-2k-1)$ choices for the coordinates; then $v$ is of type $(k+1,n-2k-1,k)$ and $d_v=n-k$.
    \item We added $1$ to a one coordinate and subtracted $1$ from a two coordinate of $u$. There are $k(n-2k-1)$ choices for the coordinates; then $v$ is of type $(k+1,n-2k-1,k)$ and $d_v=n-k$.
    \item We added $1$ to a one coordinate and subtracted $1$ from a one coordinate of $u$. There are $(n-2k-1)(n-2k-2)$ choices for the coordinates; then $v$ is of type $(k+2,n-2k-3,k+1)$ and $d_v=n-k-1$.
\end{itemize}
Thus,
\begin{equation}\label{eq:cluster_2_1}
A_1=\sum_{0 \le k<n/2}\binom{n}{2k+1}\binom{2k+1}{k}2^{-2(n-k)}\left[\frac{k(k+1)}{2}+(n-2k-1)(2k+1)+2(n-2k-1)(n-2k-2)\right].
\end{equation}

For $A_2$ we have directly
\begin{equation}\label{eq:cluster_2_2}
    A_2=\sum_{0 \le k<n/2}\binom{n}{2k+1}\binom{2k+1}{k}2^{-2(n-k)}.
\end{equation}

For $A_3$, rewrite
$$A_3= \sum_{0 \le k<n/2}\left[\binom{n}{2k+1}\binom{2k+1}{k}2^{-(n-k)}\sum_{v\in N^2_{n+1}(u)}2^{-d_v}\right]$$
Similarly to the computation of $A_1$, we distinguish the following cases on how $v$ is obtained from $u$ of type $(k+1,n-2k-1,k)$:

\begin{itemize}
    \item We added $2$ to a zero coordinate of $u$. There are $(k+1)$ choices for the coordinate; then $v$ is of type $(k,n-2k-1,k+1)$ and $d_v=n-k$.
    \item We added $1$ to two zero coordinates of $u$. There are $k(k+1)/2$ choices for the coordinates; then $v$ is of type $(k-1,n-2k+1,k)$ and $d_v=n-k+1$.
    \item We added $1$ to a zero and to a one coordinate of $u$. There are $(k+1)(n-2k-1)$ choices for the coordinates; then $v$ is of type $(k,n-2k-1,k+1)$ and $d_v=n-k$.
    \item We added $1$ to two one coordinates of $u$. There are $(n-2k-1)(n-2k-2)/2$ choices for the coordinates; then $v$ is of type $(k+1,n-2k-3,k+2)$ and $d_v=n-k-1$.
\end{itemize}

Thus,
\begin{equation}\label{eq:cluster_2_3}
A_3=\sum_{0 \le k<n/2}\binom{n}{2k+1}\binom{2k+1}{k}2^{-2(n-k)}\left[\frac{k(k+1)}{4}+(n-2k)(k+1)+(n-2k-1)(n-2k-2)\right].
\end{equation}

Combining equations \eqref{eq:cluster_2_1}, \eqref{eq:cluster_2_2}, and \eqref{eq:cluster_2_3}, completes the proof.
\end{proof}

\begin{proof}[Proof of \Cref{prop:cluster_1}~\eqref{eq:cluster_3}] 
There are finitely many types of clusters of size $3$, and the Ursell function is a constant, hence
$$|w_C(\Gam)|=O(1)\prod_{S\in \Gam}|w_C(S)|=O(1)\prod_{S\in \Gam}2^{-|\partial S|}.$$
Note that a polymer $S \in \Gamma$ contains at most $3$ vertices, and since each pair of vertices has at most one common neighbor, $|\partial S| = \sum_{u\in S}d_u-O(1)$. Hence
$$|w_C(\Gam)|= O(1)\prod_{\substack{u\in \Gam}}2^{-d_u}=O(1)2^{-\sum_{u\in \Gam}d_u},$$
where summing/multiplying over $u\in \Gam$ means we sum/multiply over all vertices $u\in S$ with $S\in \Gam$ for some $S$ with repetition if they appear in multiple polymers.

For $v\in L_{n-1}$, consider all clusters $\Gamma$ of size $3$ where $v$ is contained in some polymer $S\in \Gamma$. The set $V(\Gamma):=\bigcup_{S\in \Gamma}S$ is 2-linked with size at most $3$, and contains $v$, so the number of such sets is $n^{O(1)}$. Note also that given a set $U\subseteq L_{n-1} \cup L_{n+1}$ with $|U|\leq 3$, there are a constant number of clusters $\Gamma$ of size $3$ with $V(\Gamma)=U$.
Thus, there are at most $n^{O(1)}$ such clusters. 

Summing over all vertices in $L_{n-1}$ captures all clusters containing a vertex in $L_{n-1}$. By symmetry, this sum is equal to the same sum over vertices in $L_{n+1}$, and every cluster is captured by at least one of them. Consequently, (a justification for the third equality will follow) 
\begin{align*}
    \sum_{\|\Gam\|=3}|w_C(\Gam)|\leq 2\sum_{v\in L_{n-1}}\sum_{\substack{\Gam\ni v\\ \|\Gam\|=3}}|w_C(\Gam)|= O(1)\sum_{v\in L_{n-1}}\sum_{\substack{\Gam\ni v\\ {\|\Gam\|=3}}}2^{-\sum_{u\in \Gam}d_u}&\leq O(1)\sum_{v\in L_{n-1}}\sum_{\substack{\Gam\ni v\\ {\|\Gam\|=3}}}2^{-3d_v}\\
    &\leq n^{O(1)}\sum_{v\in L_{n-1}}2^{-3d_v},
\end{align*}
where summing over $\Gamma\ni v$ means we sum over all clusters $\Gamma$ with $v\in V(\Gamma)$. The third inequality comes from the fact that $\sum_{u\in \Gam}d_u=3d_v-O(1)$ for any $v\in \Gam$, since $\|\Gamma\|=3$ and each $u$ satisfies $|d_u-d_v|=O(1)$: $V(\Gamma)$ forms a $2$-linked graph, so $\dist(u,v)\leq 4$, and  vertices of distance $O(1)$ differ in $O(1)$ coordinates and thus their degrees differ by $O(1)$. 

Working as in the proof of \Cref{prop:cluster_1}~\eqref{eq:cluster_1}, we may rewrite the last expression as
\[n^{O(1)}2^{-3n}\sum_{0 \le k<n/2}\binom{n}{2k+1}\binom{2k+1}{k}2^{3k}=n^{O(1)}2^{-3n}\ {}_2F_1\left(\frac{1}{2}-\frac{n}{2},1-\frac{n}{2},2;32\right)=e^{-\Omega(n)}\, .
\]
{For the last equality, note that for our values of constants $a,b,c$ and for $\lam=n/2$, the dominating exponential term of the hypergeometric function is, by \eqref{eq:hyper_asymp},
$$\frac{\left(1+\frac{1}{\sqrt{z}}\right)^{\lam}}{\left(\frac{1}{\sqrt{z}}\right)^{\lam}(1+\sqrt{z})^{-\lam}}=(1+\sqrt{z})^{2\lam}=(1+4\sqrt{2})^n,$$
which is exponentially smaller than $2^{3n}$.} 
\end{proof}

\begin{proof}[Proof of \Cref{prop:cluster_1}~\eqref{eq:cluster_4}] 
With the same reasoning as in part \eqref{eq:cluster_3} we have
\[\sum_{\|\Gamma\|=2}|w_C(\Gamma)|\leq n^{O(1)}\cdot \sum_{u \in L_{n-1}}2^{-2d_u}\le n^{O(1)}\cdot 2^{-n/2}\sum_{u \in L_{n-1}}2^{-d_u} \overset{\eqref{eq:cluster_1_deg}}{=} o\left(\sum_{\|\Gamma\|=1}w_C(\Gamma)\right). \qedhere\]
\end{proof}

\section*{Acknowledgements}
MJ is supported by a UK Research and Innovation Future Leaders Fellowship MR/W007320/2. JP was supported by NSF Grant DMS-2324978, NSF CAREER Grant DMS-2443706 and a Sloan Fellowship.
MS was funded in part by the Austrian Science Fund (FWF) [10.55776/1002] and by the Bodossaki Foundation. For open access purposes, the authors have applied a CC BY public copyright license to any author accepted manuscript version arising from this submission.

\bibliographystyle{abbrv}
\bibliography{references}

\appendix

\section{Isoperimetry for $[t]^n$}\label{sec:app_isop}

We will use the following strong version of the Local Limit Theorem, \Cref{thm:llt}, due to Esseen \cite{Esseen_1945}, for random variables that have \textit{lattice distribution}. The notation and statement follows the book of Petrov~\cite{petrov_1975}.
A random variable $X$ is said to have lattice distribution if there exist real constants $a$ and $h$ with $h>0$ such that $\PP(X\in a+h\mathbb{Z})=1$. The largest $h$ for which this equality holds is called the \textit{step} of the distribution. Let 
\beq{eq:H_k} H_k(x):=(-1)^ke^{\frac{x^2}{2}}\frac{d^k}{dx^k}e^{-\frac{x^2}{2}}\enq
be the $k$-th Hermite polynomial. For a random variable $X$ with cumulants $\kap_1,\kap_2,\dots$ and variance $\sigma^2$, let
\beq{eq:q_i} q_i(x):=\frac{1}{\sqrt{2\pi}}e^{-\frac{x^2}{2}}{\sum} H_{i+2s}(x)\prod_{m=1}^i\frac{1}{k_m!}\left(\frac{\kappa_{m+2}}{(m+2)!\sigma^{m+2}}\right)^{k_m},\enq
where the sum is over all nonnegative integers $k_1,\dots, k_i$ such that $k_1+2k_2+\dots+ik_i=i$, and $s:=k_1+\dots+k_i$ (we assume $0^0=1$). We now state the Local Limit Theorem that we will use (\cite[Part II, Theorem 3]{Esseen_1945}, \cite[Chapter VII, Theorem 13]{petrov_1975}).
\begin{theorem} \label{thm:llt}
    Let $(X_n)_{n\geq1}$ be a sequence of i.i.d. random variables with mean $\mu$ and variance $\sigma^2$, and $k$ a positive integer for which $\E|X_1|^k<\infty$. Assume $X_1$ has lattice distribution with step $1$ and let $S_n:=X_1+\dots+X_n$. Then,
    $$\sigma\sqrt{n}\cdot \PP(S_n=j)=\frac{1}{\sqrt{2\pi}}e^{-x_j^2/2}+\sum_{i=1}^{k-2}\frac{q_i(x_j)}{n^{i/2}}+o\left(\frac{1}{n^{(k-2)/2}}\right)$$
    for $x_j=\frac{j-n\mu}{\sigma\sqrt{n}}$, where the error term is uniform in $j$.
\end{theorem}

A straightforward calculation (using \eqref{eq:q_i}) yields
\[q_1(x)=\frac{1}{\sqrt{2\pi}}e^{-\frac{x^2}{2}}H_3(x)\frac{\kappa_3}{6\sigma^3}; \quad \text{ and } \quad q_2(x)=\frac{1}{\sqrt{2\pi}}e^{-\frac{x^2}{2}}\left(H_6(x)\frac{\kappa_3^2}{72\sigma^6}+H_4(x)\frac{\kappa_4}{24\sigma^4}\right).\]
Applying \Cref{thm:llt} for $k=4$, we get the following corollary.
\begin{corollary}\label{lem:llt_skew}
    Let $(X_n)_{n\geq1}$ be a sequence of i.i.d. random variables with mean $\mu$, variance $\sigma^2$, $\E|X_1|^4<\infty$ and $\kappa_3(X_1)=0$. Assume $X_1$ has lattice distribution with step $1$ and let $S_n:=X_1+\dots+X_n$. Then,
    \begin{equation}\label{eqn:llt}
    \sigma\sqrt{n}\cdot\PP(S_n=j)=\frac{1}{\sqrt{2\pi}}e^{-x_j^2/2}\left(1+\frac{1}{n}H_4(x_j)\frac{\kappa_4}{24\sigma^4}\right)+o\left(\frac{1}{n}\right),
    \end{equation}
    for $x_j=\frac{j-n\mu}{\sigma\sqrt{n}}$, where the error term is uniform in $j$.
\end{corollary}
\nin A direct computation using \eqref{eq:H_k} gives that 
\beq{eq:H_4} H_4(x)=x^4-6x^2+3.\enq

Recall that $m:=\lfloor\frac{(t-1)n}{2}\rfloor$, so that $L_m$ is the (lower) middle layer of $[t]^n$. We will now finish the proof of \Cref{lem:layer_exp}, which we restate here for convenience.

\begin{lemma}\label{lem:layer_exp_app} 
There is a number $c>0$ for which the following holds: for any integer $t \ge 2$ and any $j \in [1,m]$,
$$\frac{\ell_j(t,n)}{\ell_{j-1}(t,n)}\geq 1+\frac{c}{t^2n}.$$
\end{lemma}

\begin{proof}
By \Cref{lem:logconcave}, it suffices to prove the claim for $j=m$.
    Let $X_1,X_2,\dots,X_n$ be i.i.d. random variables with $\PP(X_1=i)=\frac{1}{t}$ for all $i\in\{0,1,\dots,t-1\}$. Then $\mu=\frac{t-1}{2}$, $\sigma^2=\frac{t^2-1}{12}$, and (since $\kappa_3=\mathbb E[(X-\mathbb EX)^3]$) $\kappa_3=0$. The random vector $(X_1,\dots,X_n)$ samples elements of $[t]^n$ uniformly at random, thus
    $$\PP(S_n=j)=\frac{\ell_j(t,n)}{t^n}.$$
    So it remains to show that
    $$
    \frac{\PP(S_n=m)}{\PP(S_n=m-1)}\geq1+\frac{c}{t^2n},
    $$
 and by \Cref{lem:llt_skew},
    $$\frac{\PP(S_n=m)}{\PP(S_n=m-1)}=\frac{e^{-x_m^2/2}\left(1+\frac{1}{n}H_4(x_m)\frac{\kappa_4}{24\sigma^4}\right)+o(1/n)}{e^{-x_{m-1}^2/2}\left(1+\frac{1}{n}H_4(x_{m-1})\frac{\kappa_4}{24\sigma^4}\right)+o(1/n)}.$$
    To simplify the right-hand side of the above display, note that
    \[(x_{m-1}, x_m)=\begin{cases}(-3/(2\sigma\sqrt n),-1/(2\sigma \sqrt n)) & \text{ if $(t-1)n$ is odd;} \\ (-1/(\sigma \sqrt n), 0) & \text{ if $(t-1)n$ is even.}\end{cases}\]
    By \eqref{eq:H_4} and the fact that $x_m, x_{m-1}=o_n(1)$, $H_4(x_m), H_4(x_{m-1})=3+o(1)$. Furthermore, $\frac{\kappa_4}{24\sigma^4}$ is independent of $n$, and $H_4(x_m)\geq H_4(X_{m-1})$ as $H_4(x)$ is increasing in $[-\sqrt{3},0]$. The following observation will help us simplify the calculations.

    \begin{observation}
        If $A(n),B(n)=\Theta(1)$ then $$\frac{A(n)+o(1/n)}{B(n)+o(1/n)}=\frac{A(n)}{B(n)}+o(1/n)$$
        (where $o(1/n)$ may be negative).
    \end{observation}
\begin{proof}[Justification.]
    \[\frac{A(n)+o(1/n)}{B(n)+o(1/n)}=\frac{A(n)(1+o(1/n))}{B(n)(1+o(1/n))}=\frac{A(n)}{B(n)}(1+o(1/n)).\qedhere\]
\end{proof}
    Now, by the aforementioned facts and the observation, we have
    \begin{align*}
        \frac{\PP(S_n=m)}{\PP(S_n=m-1)}&=\frac{e^{-(x_m^2-x_{m-1}^2)/2}\left(1+\frac{1}{n}H_4(x_m)\frac{\kappa_4}{24\sigma^4}\right)}{1+\frac{1}{n}H_4(x_{m-1})\frac{\kappa_4}{24\sigma^4}}+o(1/n)\\
        &\geq e^{\frac{c}{\sigma^2n}}+o(1/n) \geq 1+\frac{c}{t^2n},
    \end{align*}
    where $c$ is an absolute constant. This yields the conclusion.
\end{proof}
When $(t-1)n$ is odd (i.e., $t$ is even and $n$ is odd), $[t]^n$ has two middle layers. If one wishes to extend the methods in this paper to $[t]^n$ for general $t$, we anticipate that one would need isoperimetric inequalities between the two middle layers as in \Cref{prop:isoperim_mid} below. We note that this lemma doesn't follow from \Cref{lem:layer_exp} or its proof argument, so it requires a separate proof. 

Recall that $m:=\lfloor\frac{(t-1)n}{2}\rfloor$, so when $(t-1)n$ is odd, $L_{m}$ and $L_{m+1}$ are the two middle layers of $[t]^n$.
\begin{lemma}\label{prop:isoperim_mid}
    There is a number $c>0$ for which the following holds: for any integers $t \ge 4$ and $n \ge 1$ such that $(t-1)n$ is odd, for every $S\subseteq L_{m+1}(t,n)$ with $|S|\leq \frac{1}{2}\ell_{m+1}(t,n)$, $|N^-(S)|\geq \left(1+\frac{c}{t^2n}\right)|S|$.
\end{lemma}
The following trivial observation will be used frequently in the rest of the section. Recall from \Cref{sec:isoper} that for $S \sub [t]^n$, $S^j$ denotes the set of elements of $S$ whose first coordinate is $j$; $P(S)$ is the image of $S$ under the mapping $P$ where $P:[t]^n \rightarrow [t]^{n-1}$ is the operation of removing the first coordinate of the given element of $[t]^n$. In the notation below, we will often suppress dependency on $t$.
\begin{observation}\label{obs:bij}
Let $S\subseteq [t]^n$.
\begin{enumerate}[(a)]
    \item For every $S\subseteq L_k^j(n)$, we have $|S|=|P(S)|$. In particular, $|L_k^j(n)|=|L_{k-j}(n-1)|=\ell_{k-j}(n-1)$.
    \item If $S=S^0$, then $|N^-(S)|=|N^-(P(S))|$; if $S=S^{t-1}$, then $|N^+(S)|=|N^+(P(S))|$.
\end{enumerate}
\end{observation}
We first prove some supporting claims assuming $(t-1)n$ is odd. Recall \Cref{def:compression}. We say $S$ is compressed if $S=C(S)$.
\begin{claim}\label{claim:half_coord}
    Let $S\subseteq L_{m+1}(n)$ be a compressed set with $|S|\leq \frac{1}{2}\ell_{m+1}(n)$. Then $S\subseteq \bigcup_{j=0}^{t/2}L_{m+1}^j(n)$.    
\end{claim}
\begin{proof}
Since $S$ is compressed, $S\subseteq \bigcup_{j=0}^{k}L_{m+1}^j(n)$ for some $k$. We prove that $k\leq t/2$. To this end, setting $\ell_{m+1}^j(n):=|L_{m+1}^j(n)|$, it suffices to show that
$$\sum_{j=0}^{t/2}\ell_{m+1}^j(n)\geq\frac{1}{2}\ell_{m+1}(n).$$
Note that $\ell^0_{m+1}(n)=\ell_{m+1}(n-1)$, and for every $1\leq j\leq t/2-1$,
\[\ell_{m+1}^j(n)=\ell_{m+1-j}(n-1)=\ell_{m+1-(t-j)}(n-1)=\ell_{m+1}^{t-j}(n)\]
where we use \Cref{obs:bij} (a) for the first and the last equality; for the second equality, observe that the (unique) middle layer of $[t]^{n-1}$ is $L_{m-\frac{t}{2}+1}(n-1)$ while $L_{m+1-j}(n-1)$ and $L_{m+1-(t-j)}(n-1)$ are symmetric about the middle layer. Hence,
\[    2\sum_{j=0}^{t/2}\ell_{m+1}^j(n)=\ell_{m+1}^0(n)+\ell_{m+1}^{t/2}(n)+\sum_{j=0}^{t-1}\ell_{m+1}^j(n) =\ell_{m+1}^0(n)+\ell_{m+1}^{t/2}(n)+\ell_{m+1}(n) \ge \ell_{m+1}(n).\qedhere
\]
\end{proof}

\begin{claim}\label{claim:frac}
    Let $a,b,c,d,x,y\in \mathbb{R}^+$ such that $\frac{a}{b}\geq\frac{c}{d}$, $\frac{x}{y}\geq\frac{c}{d}$, $x\leq c$ and $y\leq d$. Then
    \beq{ineq:ax}\frac{a+x}{b+y}\geq\frac{a+c}{b+d}.\enq
\end{claim}
\begin{proof}
    If $y=d$, then $x=c$ and the claim trivially holds, so assume $d-y>0$. 
    
    Note that \eqref{ineq:ax} is equivalent to $a(d-y)-b(c-x)+xd-cy\geq0$. Since $\frac{x}{y}\geq\frac{c}{d}$, we have $xd-cy\geq0$. Also, given $d-y>0$, $a(d-y)\geq b(c-x)$ is equivalent to $\frac{a}{b}\geq\frac{c-x}{d-y}$.    
    But it is easy to see that $\frac{c}{d}\geq \frac{c-x}{d-y}$ when $\frac{x}{y}\geq\frac{c}{d}$. Since $\frac{a}{b}\geq\frac{c}{d}$, the claim follows. 
\end{proof}

\begin{proof}[Proof of \Cref{prop:isoperim_mid}]
By \Cref{prop:CL}, we may assume that $S$ is compressed. Then $S=\bigcup_{j=0}^{k-1}L_{m+1}^j(n) \cup X$ for some $X\subseteq L_{m+1}^k(n)$, and thus $N^-(S)=\bigcup_{j=0}^{k-1}L_m^j(n) \cup (N^-(X)\cap L_m^k(n))$.
By \Cref{claim:half_coord}, $k\leq \frac{t}{2}$. 

By \Cref{obs:bij} (a), $\ell^s_j(n)=\ell_{j-s}(n-1)$ for all $j,s$. Since $X\subseteq L_{m+1}^k(n)$, the same observation implies $|X|=|P(X)|$, and, with $Y:=P(X)~(\sub L_{m-k+1}(n-1))$, $|N^-(X)\cap L_m^k(n)|=|N^-(Y)|$.
Then again by \Cref{obs:bij}, now $|N^-(S)|/|S|\ge 1+\frac{c}{t^2n}$ is equivalent to
\begin{equation}\label{eqn:isop_mid}
    \frac{\ell_{m}(n-1)+\ell_{m-1}(n-1)+\dots+\ell_{m-k+1}(n-1)+|N^-(Y)|}{\ell_{m+1}(n-1)+\ell_{m}(n-1)+\dots+\ell_{m-k+2}(n-1)+|Y|}\geq1+\frac{c}{t^2n}.
\end{equation}

Note that for every $0\leq s\leq k-1~(\le t/2-1)$,
$$
    \frac{\ell_{m-s}(n-1)}{\ell_{m-s+1}(n-1)}\geq 1+\frac{c}{t^2n}
$$
since $L_{m-s}(n-1)$ is above or equal to $L_{m-\frac{t}{2}+1}(n-1)$, the middle layer of $[t]^{n-1}$ (we use \Cref{lem:layer_exp} and the symmetry along the middle layer). If $k\le t/2-1$, then $L_{m-k+1}(n-1)$ is strictly above the middle layer, so the normalized matching property \eqref{eq:nmp} gives
$$\frac{|N^-(Y)|}{|Y|}\geq \frac{\ell_{m-k}(n-1)}{\ell_{m-k+1}(n-1)}\geq1+\frac{c}{t^2n},$$
which proves \eqref{eqn:isop_mid}. 

Now the only remaining case is $k=t/2$. We rewrite the left-hand side of \eqref{eqn:isop_mid} as $(A+|N^-(Y)|)/(B+|Y|)$ by setting $A:=\sum_{j={m-k+1}}^m\ell_j(n-1)$ and $B:=\sum_{j=m-k+2}^{m+1}\ell_j(n-1)$. We now apply \Cref{claim:frac} for $a:=A$, $b:=B$, $x:=|N^-(Y)|$, $y:=|Y|$, $c:=\ell_{m-\frac{t}{2}}(n-1)$ and $d:=\ell_{m-\frac{t}{2}+1}(n-1)$. (To see the assumptions of \Cref{claim:frac} hold: we have $\frac{a}{b}> 1 > \frac{c}{d}$ and $\frac{x}{y}\geq \frac{c}{d}$ by the normalized matching property; $x\leq c$ and  $y\leq d$ are trivial.)  Then, the left-hand side of \eqref{eqn:isop_mid} is at least
$$\frac{A+\ell_{m-\frac{t}{2}}(n-1)}{B+\ell_{m-\frac{t}{2}+1}(n-1)}\geq 1+\frac{\ell_{m-\frac{t}{2}}(n-1)-\ell_{m+1}(n-1)}{t\ell_{m-\frac{t}{2}+1}(n-1)}=:1+E(n).$$
We now apply \Cref{lem:llt_skew} to estimate $E(n)$. Note that the layers involved with $E(n)$ are in $[t]^{n-1}$, so (in applying \Cref{lem:llt_skew}) $x_j=\frac{j-(n-1)\mu}{\sigma\sqrt{n-1}}$ where $\mu=\frac{t-1}{2}$ and $\sigma^2=\frac{t^2-1}{12}.$ A straightforward calculation yields
$$x_{m-\frac{t}{2}}=-\frac{1}{\sigma\sqrt{n-1}}, \hspace{1cm} x_{m+1}=\frac{t}{2\sigma\sqrt{n-1}}, \hspace{1cm} x_{m-\frac{t}{2}+1}=0.$$
Note that $H_4(x)$ is an even function (see \eqref{eq:H_4}) and it's decreasing in $[0,\sqrt{3}]$, thus $H_4(x_{m+1})\leq H_4(x_{m-\frac{t}{2}})$. This yields
\begin{equation}
    E(n) \geq \frac{\left(1+\frac{1}{n}H_4(\frac{1}{\sigma\sqrt{n-1}})\frac{\kappa_4}{24\sigma^4}\right)\left(e^{-\frac{1}{2\sigma^2(n-1)}}-e^{-\frac{t^2}{8\sigma^2(n-1)}}\right)+o(1/n)}{t\left(1+\frac{3}{n}\frac{\kappa_4}{24\sigma^4}+o(1/n)\right)}.
\end{equation}
Using that $e^{-x}=1-x+o(x)$ when $x\rightarrow 0$  and $\sigma^2=\Theta(t^2)$ yields
$$E(n)\geq\frac{c}{tn} \geq\frac{c}{t^2n}$$
and the proof is concluded.
\end{proof}

\section{Proofs for containers}

\subsection{Proof of \Cref{lem:phi}} \label{app:lem:phi}

For a bipartite graph with bipartition $U \cup W$ we say that $W'\sub W$ \emph{covers} $U$ if each $u \in U$ has a neighbor in $W'$.

    \begin{lemma}[Lov\'asz \cite{Lov75}, Stein \cite{Ste74}]\label{lem:cover}
        Let $ \Sigma$ be a bipartite graph with bipartition $U \cup W$, where $|N(u)|\ge x$ for each $u \in U$ and $|N(w)|\le y$ for each $w \in W$. Then there exists some $W' \sub W$ that covers $U$ and satisfies
\[|W'|\le\frac{|W|}{ x}\cdot(1+\ln y).\]
    \end{lemma}

Set $p=\frac{10\Delta \log d}{\phi \delta d}$.

\begin{lemma}\label{lem:phi.approx.premier}
For any $A \in \cG(a,g,\gk)$, there is a set $T_0\subseteq N(A)~(=G)$ such that
        \beq{eqn:lem_aux_1} |T_0|\leq 3gp; \enq
        \beq{eqn:lem_aux_2} |\nabla(T_0,X\setminus[A])|\leq 3\ka p; \text{ and} \enq
        \beq{eqn:lem_aux_3} |G^\phi\setminus N(N_{[A]}(T_0))|\leq 3g/d^{10}. \enq
    \end{lemma}

    \begin{proof}
Let $\bT \sub G$ be a random subset that takes each element of $G$ independently at random with probability $p$. Then we have $\E|\bT|=gp$ and $\E(|\nabla(\bT, X \setminus [A])|)=\gk p.$ Furthermore, for any $u\in G^\phi$ we have $|N(N_{[A]}(u))|\geq \frac{\phi \gd d}{\Delta}$, which implies
        \begin{align*}
            \E\left(|G^\phi\setminus N(N_{[A]}(\bT))|\right)&=\sum_{u\in G^\phi} \PP(u\not\in N(N_{[A]}(\bT)))=\sum_{u\in G^\phi}\PP(N(N_{[A]}(u))\cap \bT=\emptyset)\\
            &\leq g(1-p)^{\frac{\phi \gd d}{\Delta}}\leq g/d^{10}.
        \end{align*}
        Hence, by Markov's inequality, there exists a set $T_0 \sub G$ satisfying the desired properties.
    \end{proof}

    Set 
    \[T_0'=G^\phi\setminus N(N_{[A]}(T_0)), \quad \Omega=\nabla(T_0,X\setminus[A]), \quad L=N(N_{[A]}(T_0))\cup T_0'.\]
    Let $T_1\subseteq G\setminus L$ be a cover of minimum size of $[A]\setminus N(L)$ in the graph induced by $(G\setminus L)\cup ([A]\setminus N(L))$. Then $F':=L\cup T_1$ is a $\phi$-approximation of $A$. Our goal now is to bound the number of sets $F'$ produced this way.

    Note that (\ref{eqn:lem_aux_1}), (\ref{eqn:lem_aux_2}) and (\ref{eqn:lem_aux_3}), respectively, yield
    \[|T_0|\leq O\left(\frac{g\Delta\log d}{\phi \gd d}\right), \quad |\Omega|\leq O\left(\frac{\kap \Delta \log d}{\phi \delta d}\right), \quad |T_0'|\leq O\left(\frac{g}{d^{10}}\right).\]
    As $T_1$ is a minimum cover of $[A] \setminus N(L)$, we can use Lemma \ref{lem:cover} to bound its size. First we need to bound $|G\setminus L|$. Each vertex $u$ of $G\setminus L$ is in $G\setminus G^\phi$, so $u$ contributes at least $d(u)-\phi \ge \delta d-\phi$ edges in $\nabla(G,X\setminus [A])$; but this set has size $\kappa$, hence $|G\setminus L|\leq \frac{\kap}{\delta d-\phi}$. Furthermore, by our degree assumption on $\Sigma$, $d_{[A]\setminus N(L)}(u)\le d$ for each $u\in G\setminus L$ and $d_{G\setminus L}(v)\ge \delta d$ for each $v\in[A]\setminus N(L)$. Hence, Lemma \ref{lem:cover} implies 
    $$|T_1|\leq \frac{\kap}{\delta d(\delta d-\phi)}(1+\ln d)=O\left(\frac{\kappa\log d}{\delta d(\delta d-\phi)}\right).$$
Let $T:=T_0 \cup T_0' \cup T_1$.
\begin{claim}[\cite{balogh2024intersecting}, Claim 2]
    $T$ is an 8-linked subset of $Y$.
\end{claim}
Now, since $N(N_{[A]}(T_0))$ is determined by $T_0$ and $\gO$, it follows that $F'$ is determined by $T_0, \gO, T'_0,$ and $T_1$. Let $\cV$ be the collection of all sets $F'$ that is constructed from some tuple $(T_0, T_0', T_1, \gO)$.

Note that $|T|=O\left(\frac{g\Delta\log d}{\phi \gd d}+\frac{\gk\log d}{\gd d(\gd d-\phi)}\right)$, so by \Cref{lem:linked_count} there are at most
\[|Y|\exp\left[O\left(\frac{g\Delta \log^2 d}{\phi\gd d}+\frac{\gk\Delta\log^2d}{\gd d(\gd d-\phi)}\right)\right]\]
choices for $T$. Trivially, the choices for $T_0, T_0', T_1$ are bounded by $2^{|T|}$. Finally, as $\gO\sub \nabla(G, X \setminus {[A]})$ which is a set of size $\gk$, the number of choices for $\gO$ is at most
$$\binom{\kap}{\leq O\left(\frac{\kap\Delta \log d}{\phi \gd d}\right)}\leq 2^{O\left(\frac{\kap \Delta \log^2d}{\phi\gd d}\right)}.$$
In conclusion, the total number of choices for $(T_0, T_0', T_1, \gO)$ is at most
\[|Y|\exp\left[\frac{\Delta}{\gd}O\left(\frac{g\log^2 d}{\phi d}+\frac{\gk\log^2d}{d(\gd d-\phi)}+\frac{\gk\log^2 d}{\phi d}\right)\right].\]

\subsection{Proof of \Cref{lem:psi}} \label{app:lem:psi}

Let $A\in \cG(a,g, \kappa)$ and $F'$ be a $\phi$-approximation as produced in \Cref{lem:phi}. We will use the algorithm below to produce a $\psi$-approximation $(S,F)$ for $A$ given an input $(A,F')$, and then bound the possible outputs of this algorithm. Fix a total order $\prec$ on $V(\Sigma)$.

\begin{quote}
    \nin \textbf{Input.} $A \in \cG(a,g, \kappa)$ and its $\varphi$-approximation $F' \in \cV(a,g,\kappa, \varphi)$

    \nin \textbf{Step 1.}  If $\{u \in [A]:d_{G\setminus F'}(u)>\psi\} \ne \emptyset,$ pick the smallest $u$ (in $\prec$) in this set and update $F'$ by $F' \leftarrow F' \cup N(u).$ Repeat this until $\{u \in [A]:d_{G \setminus F'}(u)>\psi\}=\emptyset.$ Then set $F''=F'$ and $S''=\{u \in X:d_{F''}(u) \ge d(u)-\psi\}$ and go to Step 2.

    \nin \textbf{Step 2.} If $\{v \in Y \setminus G: d_{S''}(v)>\psi\} \ne \emptyset$, pick the smallest $v$ (in $\prec$) in this set and update $S''$ by $S'' \leftarrow S'' \setminus N(v).$ Repeat this until $\{v \in Y \setminus G:d_{S''}(v)>\psi\}=\emptyset.$ Set $S=S''$ and $F=F'' \cup \{v \in Y:d_{S}(v)>\psi\}$ and stop.

    \nin \textbf{Output.} $(F, S)$
\end{quote}

    It is easy to see that the resulting $(F, S)$ is a $\psi$-approximation for $A \in \cG(a,g)$. (See, e.g., the second paragraph of the proof of \cite[Lemma 3.1]{jenssen2024refined}.) We show that the number of outputs for each $F'$ is at most the right-hand side of \eqref{psi_bd}. 

    \nin \textit{Cost for Step 1.} Initially, $|G \setminus F'| \le \gk/(\delta d-\varphi)$ (each vertex $w \in G \setminus F'$ is in $G \setminus G^\varphi$ and so contributes at least $d(w)-\varphi \ge \delta d-\varphi$ edges to $\nabla(G, X \setminus [A])$, a set of size $\gk$). Each iteration in Step 1 removes at least $\psi$ vertices from $G \setminus F'$ and so there can be at most $\gk/((\delta d-\varphi)\psi)$ iterations. The $u$'s in Step 1 are all drawn from $[A]$ and hence $N(F')$, a set of size at most $dg$. So the total number of outputs for Step 1 is at most \[{dg \choose \le \gk/((\gd d-\varphi)\psi)}.\]

    \nin \textit{Cost for Step 2.} Each $u \in S'' \setminus [A]$ contributes more than $d(u)-\psi \ge \delta d-\psi$ edges to $\nabla(G, X \setminus [A])$, so initially $|S'' \setminus [A]| \le \gk/(\delta d-\psi).$ Each $v$ used in Step 2 reduces this by at least $\psi,$ so there are at most $\gk/((\delta d-\psi)\psi)$ iterations. Each $v$ is drawn from $N(S''),$ a set which is contained in the second neighborhood of $F'$ (since $S'' \sub N(F'')$ by definition) and so has size at most $d^2 g.$
    So the total number of outputs for Step 2 is at most
    \[{d^2 g \choose \le \gk/((\delta d-\psi)\psi)}. \qedhere \]

\subsection{Proof of \Cref{lem:phi'}}\label{app:lem:phi'}

Write $\ov{[B]}$ for $L_{n-2} \setminus [B]$ and define 
\[\cG_{n-2}(b,h, \gk, S):=\{B \sub L_{n-2}:|[B]|=b, |N(B)|=h, |\nabla(N(B), \ov{[B]})|=\gk, N(B) \sub S\}.\]

By repeating the proof of \Cref{lem:phi.approx.premier} with $p=20\log n/\phi n$ (recall that $\Delta=1$ and $\gd=1/2-o(1)$), one can show that for any $B \in \cG_{n-2}(b,h,\gk,S)$, there is a set $T_0 \sub N(B)$ such that
\[|T_0|\le 3hp; \quad e(T_0, \ov{[B]}) \le 3\kappa p; \quad \text{and} \quad  |N(B)^\varphi \setminus N(N_{[B]}(T_0))| \le 3h/n^{10}.\]
Let
\[T_0'=N(B)^\varphi \setminus N(N_{[B]}(T_0)), \quad L=T_0' \cup N(N_{[B]}(T_0)), \quad \Omega=\nabla(T_0, \ov{[B]}).\]
Let $T_1 \sub N(B) \setminus L$ be a minimal set that covers $[B] \setminus N(L)$ in the subgraph induced by $[B] \setminus N(L) \cup N(B) \setminus L$. Let $F':=L \cup T_1$. Then $F'$ is a $\phi$-approximation of $B$. Also, since $N(N_{[B]}(T_0))$ is determined by $T_0$ and $\Omega$, $F'$ is determined by $T_0$, $\Omega$, $T_0'$, and $T_1$. Note that
\[|N(B) \setminus L|(n/2-\phi) \le e(N(B), \ov{[B]})=\gk,\]
so by Lemma \ref{lem:cover} (with $U=[B] \setminus N(L)$, $x= n/2$, $W=N(B) \setminus L$ and $ y = n$), we have
\[|T_1| \le \frac{|N(B) \setminus L|}{n/2}(1+\ln n) \le \frac{3\gk\ln n}{n(n/2-\phi)}.\]
Finally, using the fact that $T_0, T_0', T_1$ are subsets of $N(A) \sub S$ and $\Omega$ is a subset of $\nabla(T_0)$, the number of choices for $T_0, T_0', T_1$ and $\Omega$ is at most
\[\begin{split}
    &\binom{s}{\le 3hp}\binom{s}{\le 3h/n^{10}}\binom{s}{\le \frac{3\gk\ln n}{n(n/2-\phi)}}\binom{3hpn}{\le 3\gk p}\\
    &=\exp\left[O\left(\frac{h\log n}{n\phi}\log\left(\frac{sn}{h}\right)+\frac{\kappa \log n}{n(n/2-\phi)}\log\left(\frac{sn}{\gk}\right)+\frac{\gk\log n}{\phi n}\log\left(\frac{hn}{\gk}\right)\right)\right]\\
    &\le \exp\left[O\left(\frac{h\log n}{n\phi}\log\left(\frac{sn}{h}\right)+\frac{\kappa \log n}{n(n/2-\phi)}\log\left(\frac{sn}{\gk}\right)+\frac{\gk\log^2 n}{\phi n}\right)\right]\\
    &\le |S|\exp\left[O\left(\frac{h\log^2 n}{n\phi}+\frac{\kappa \log^2 n}{n(n/2-\phi)}+\frac{\gk\log^2 n}{\phi n}\right)\right]\end{split}
\]
where the last inequality uses \eqref{incr} and the facts that $\gk \le t'n$ (\Cref{prop:t.kappa}) and $t'=\gO(h/n)$ (\Cref{prop:isoperim_up}(b)).

Finally, by applying the above construction for all $\gk\le t'n$, we obtain the desired $\cV(b,h,S)$, and it satisfies
\[\begin{split}|\cV|&\le \sum_{\gk\le t'n}|S|\exp\left[O\left(\frac{h\log^2 n}{n\phi}+\frac{\kappa \log^2 n}{n(n/2-\phi)}+\frac{\gk\log^2 n}{\phi n}\right)\right]\\
&\le t'n|S|\exp\left[O\left(\frac{h\log^2 n}{n\phi}+\frac{t' \log^2 n}{n/2-\phi}+\frac{t'\log^2 n}{\phi }\right)\right].\end{split}\]

\end{document}